\documentclass{article}

\usepackage{latexsym,amsmath,amsfonts,amscd, amsthm, dsfont}
\usepackage{bm,color}
\usepackage{epsfig,verbatim,epstopdf,graphics,graphicx}
\usepackage{subfigure,changebar,multirow}

\usepackage{algorithmic}
\usepackage[bottom]{footmisc}
\usepackage{yhmath,booktabs,tikz,verbatim}
\usepackage{authblk,enumitem}
\usetikzlibrary{arrows,backgrounds,snakes,shapes}

\usepackage[sort]{cite}

\numberwithin{equation}{section}

\graphicspath{{./}{./figure/}}
\allowdisplaybreaks

\newtheoremstyle{plainNoItalics}{}{}{\normalfont}{}{\bfseries}{.}{ }{}

\theoremstyle{plain}
\newtheorem{thm}{Theorem}[section]

\theoremstyle{plainNoItalics}

\newtheorem{lem}[thm]{Lemma}

\newtheorem{rem}[thm]{Remark}

\newtheorem{exa}[thm]{Example}

\newcommand{\RR}{\mathbb R}
\newcommand{\NN}{\mathbb N}

\newcommand{\mD}{{\mathcal D}}
\newcommand{\mP}{{\mathcal P}}
\newcommand{\mO}{{\mathcal O}}

\newcommand{\be}{\begin{eqnarray}}
\newcommand{\ee}{\end{eqnarray}}
\newcommand{\beno}{\begin{eqnarray*}}
\newcommand{\eeno}{\end{eqnarray*}}

\makeatletter

\newcommand{\Rmnum}[1]{\expandafter\@slowromancap\romannumeral #1@}
\makeatother

\begin{document}

\title{A Kernel-Based Explicit Unconditionally Stable Scheme for Hamilton-Jacobi Equations on Nonuniform Meshes}
\date{}
\author{Andrew Christlieb${}^{1,*}$, William Sands${}^{1,\dagger}$ and Hyoseon Yang${}^{1,\ddagger}$}
\footnotetext[1]{Department of Computational Mathematics, Science and Engineering, Michigan State University, East Lansing, MI, 48824, United States. $^*$christli@msu.edu, ${}^\dagger$sandswi3@msu.edu, and ${}^\ddagger$hyoseon@msu.edu; corresponding author.}
\maketitle

\begin{abstract}
	In \cite{christlieb2019kernel},
	the authors developed a class of high-order numerical schemes for the Hamilton-Jacobi (H-J) equations, which are unconditionally stable, yet take the form of an explicit scheme.
	This paper extends such schemes, so that they are more effective at capturing sharp gradients, especially on nonuniform meshes.
	In particular, we modify the weighted essentially non-oscillatory (WENO) methodology in the previously developed schemes by incorporating an exponential basis and adapting the previously developed nonlinear filters used to control oscillations. The main advantages of the proposed schemes are their effectiveness and simplicity, since they can be easily implemented on higher-dimensional nonuniform meshes. We perform numerical experiments on a collection of examples, including H-J equations with linear, nonlinear, convex and non-convex Hamiltonians. To demonstrate the flexibility of the proposed schemes, we also include test problems defined on non-trivial geometry.
\end{abstract}


\pagestyle{myheadings} \thispagestyle{plain}\markboth{}{}
\section{Introduction}

In this paper, we propose a class of high-order, weighted essentially non-oscillatory numerical schemes for approximating the viscosity solution to the Hamilton-Jacobi (H-J) equation 
\begin{align}\label{eq:HJ}
\begin{cases}
\phi_t + H(\nabla  \phi)  = 0, \quad x  \in \mathbb{R}^d \\
\phi(x,0)  = \phi_0(x),
\end{cases}
\end{align}
where $\phi=\phi(x,t)$ is a scalar function and $H$ is a Lipschitz continuous Hamiltonian. 
The H-J equations play a significant role among many fields, including optimal control, geometric optics, differential games, computer vision and image processing, as well as variational calculus. 
It is well known that as time evolves, the H-J equations develop continuous solutions, of which, associated derivatives might be discontinuous, even for smooth initial conditions. If the solution is redefined in a weak sense, regularity conditions on the function $\phi$ can be relaxed; however, such solutions may not be unique. To identify the unique, physically relevant solution, the concept of vanishing viscosity was introduced \cite{crandall1983viscosity, crandall1984some}. In subsequent papers,\cite{crandall1984two, soug}, authors addressed the convergence of general approximation schemes to the viscosity solution of \eqref{eq:HJ}.

There have been many numerical schemes developed to solve the H-J equations. Methods among the existing literature include the essentially non-oscillatory (ENO) schemes \cite{osher1988fronts, osher1991high}, weighted ENO (WENO) schemes \cite{jiang2000weighted, zhang2003high}, Hermite WENO schemes \cite{qiu2005hermite, qiu2007hermite,zhu2013hermite, zheng2017finite}, as well as discontinuous Galerkin methods \cite{hu1999discontinuous, lepsky2000analysis, cheng2007discontinuous, yan2011local, cheng2014new}. These schemes are typically categorized within the Method of Lines (MOL) framework, in which the spatial variable is discretized first, then the resulting initial value problems (IVPs) are solved by coupling with a suitable time integrator. This work takes an alternative approach: First, discretization is completed on the temporal variable, then, the resulting boundary value problems (BVPs) are solved at discrete time levels. To solve the BVPs, the continuous operator (in space) is inverted analytically, using an integral solution. We refer to this approach as the Method of Lines Transpose (MOL$^T$), which is also known as Rothes's method \cite{schemann1998adaptive, salazar2000theoretical, causley2014method}. These methods are formally matrix-free, in the sense that there is no need to solve linear systems at each time step. Moreover, this integral solution extends the so-called domain-of-dependence, so that the method does not suffer from a CFL restriction. 
The kernel used in this formulation also exhibits pleasant numerical properties with several developments. To approximate the integral equations in BVP, 
the fast multipole method(FMM) solved the heat, Navier-Stokes and linearized Poisson-Boltmann equation in \cite{fmm0,fmm1},
Fourier-continuation alternating-direction(FC-AD) algorithm yields unconditionally stability  from $\mathcal{O}(N^2)$ to  $\mathcal{O}(N\log N)$ \cite{fc-ad1,fc-ad2}
and Causley {\it et al.}\cite{causley2013method} reduces the computational complexity of the method from $\mathcal{O}(N^2)$ to $\mathcal{O}(N)$. A variety of schemes, based on the MOL$^T$ formulation, have been developed for solving a range of time-dependent PDEs, including the wave equation \cite{causley2014method}, the heat equation (e.g., the Allen-Cahn equation \cite{causley2016method} and Cahn-Hilliard equation \cite{causley2017method}), Maxwell's equations \cite{cheng2017asymptotic}, and the Vlasov equation \cite{christlieb2016weno}.

Recent work on the MOL$^T$ has involved extending the method to solve more general nonlinear PDEs, for which an integral solution is generally not applicable. This work includes the nonlinear degenerate convection-diffusion equations \cite{christlieb2017kernel}, as well as the H-J equations \cite{christlieb2019kernel}. The key idea of these papers involved exploiting the linearity of a given \textit{differential operator}, rather than requiring linearity in the underlying equations. This allowed derivative operators in the problems to be expressed through kernel representations developed for linear problems. Formulating applicable derivative operators in this way ultimately facilitated the stability of the schemes, since a global coupling was introduced through the integral operator. As part of this embedding process, a kernel parameter $\beta$ was introduced, and through a careful selection, was shown to yield schemes which are A-stable. Remarkably, it was shown that one could couple these representations for the derivative operators with an explicit time-stepping method, such as the strong-stability-preserving Runge-Kutta (SSP-RK) methods \cite{gottlieb2001strong} and still obtain schemes which maintain unconditional stability \cite{christlieb2017kernel,christlieb2019kernel}. To address shock-capturing and control non-physical oscillations, the latter two papers introduced quadrature formulas based on WENO reconstruction, along with a nonlinear filter. 



This paper seeks to extend the work in \cite{christlieb2017kernel, christlieb2019kernel} to the H-J equations \eqref{eq:HJ} defined on non-uniformly distributed spatial domains. In particular, several improvements are given. First, we develop the MOL$^T$ for mapped grids using a general coordinate transformation function, which allows for a non-uniform distribution of grid points. We show that, with this mapping, our numerical scheme is able to preserve the conservation property for the derivative of the solution to the H-J equation. We also describe a novel WENO-based quadrature for the spatial discretization, which uses a basis consisting of exponential polynomials, to improve the shock capturing capabilities of the method. Another difference in this paper, compared to our previous work on H-J equations, is that we propose a different nonlinear filter, which, we believe, is more effective at minimizing oscillations in the derivative of the solution to the PDE \eqref{eq:HJ}.

The paper is organized as follows. We first review the kernel-based representations for first and second order derivative operators, and address boundary conditions for both periodic and non-periodic problems, in Section 2. In Section 3, we present our numerical scheme for H-J equations on nonuniform grids with an algorithm flowchart. A collection of numerical examples is presented to demonstrate the performance of the proposed method in Section 4. In Section 5, we conclude the paper with some remarks and directions of future work.

\section{Review for the approximation of differential operators}

We start with a brief review on construction of derivative operators using the methodologies proposed in \cite{christlieb2017kernel,christlieb2019kernel}. The second order derivative, e.g., $\partial_{xx}$, shall described first, as it will be used in the representation of first derivatives. Note that representations are formed using 1D examples, but the line-by-line approach allows us the reuse these expressions, with an appropriate swapping of the direction. 

\subsection{Second order derivative $\partial_{xx}$}
\label{sec:second order derivative}
In this section we will develop an approximation to $\partial_{xx}$ based on a fast kernel method.  The starting point is a Helmholtz operator of the ``right sign'', meaning that the inverses is represented by a ``compact'' kernel.  Here ``compact'' refers to a kernel that is represented as a function instead of an infinite sum. This representation is used to build an approximation to $\partial_{xx}$.

Motivated by work done for parabolic equations (see e.g., \cite{causley2016method,causley2017method}), we define the differential operator
\begin{equation}
    \label{eq:L_0 operator definition}
    \mathcal{L}_{0} := \mathcal{I} - \frac{1}{\alpha^{2}} \partial_{xx}, \quad x \in [a,b],
\end{equation}
where $\mathcal{I}$ is the identity operator, and $\alpha$ is a positive constant, which shall be specified later. We now suppose that there are two functions $w(x)$ and $v(x)$, which satisfy the equation
\begin{equation}
    \label{eq: second derivative BVP}
    \left( \mathcal{I} - \frac{1}{\alpha^{2}} \partial_{xx} \right) w(x) = v(x).
\end{equation}
Noting that this is a linear equation of the form
\begin{equation*}
    \mathcal{L}_{0}[w;\alpha](x) = v(x),
\end{equation*}
it follows that the solution can be obtained through an analytic inversion of the operator $\mathcal{L}_{0}$:
\begin{equation*}
    w(x) = \mathcal{L}_{0}^{-1}[v;\alpha](x).
\end{equation*}
Written more explicitly, the expression for $w(x)$ can be determined to be
\begin{equation}
    \label{eq: Inversion of second derivative}
    w(x) = I_{0}[v; \alpha](x) + A_{0}e^{-\alpha (x - a)} + B_{0}e^{-\alpha (b - x)},
\end{equation}
where 
\begin{equation}
    \label{eq: Second derivative convolution integral}
    I_{0}[v; \alpha](x) := \frac{\alpha}{2} \int_{a}^{b} e^{-\alpha \lvert x - s \rvert}v(s) \, ds,
\end{equation}
is a convolution integral and the constants $A_{0}$ and $B_{0}$ are determined by boundary conditions. If the PDE is linear, e.g., the heat equation, then \eqref{eq: Inversion of second derivative} is a valid expression for the update, and $A_{0}$ and $B_{0}$ can be determined using the boundary conditions specified by the problem. Otherwise, they will need to be carefully prescribed to maintain consistency. We will address this issue in Sections \ref{sec:periodic} and \ref{sec:non_periodic}.

To develop a suitable expression for the second derivative, we introduce the related operator $\mathcal{D}_{0}$, which is defined as
\begin{equation}
    \label{eq:D0 operator}
    \mathcal{D}_{0} = \mathcal{I} - \mathcal{L}_{0}^{-1}.
\end{equation}
Through some algebraic manipulations, one can write an alternative definition for $\mathcal{L}_{0}$ in terms of $\mathcal{D}_{0}$, i.e.,
\begin{equation*}
    \mathcal{L}_{0} = \left( \mathcal{I} - \mathcal{D}_{0} \right)^{-1}.
\end{equation*}
If the operator norm for $\mathcal{D}_{0}$ is bounded by unity, then, using the definition \eqref{eq:L_0 operator definition}, we can express the second derivative operator as a Neumann series:
\begin{equation*}
    \frac{1}{\alpha^{2}} \partial_{xx} = \mathcal{I} - \mathcal{L}_{0} = \mathcal{L}_{0} \left( \mathcal{L}_{0} - \mathcal{I} \right) = - \mathcal{D}_{0} \left( \mathcal{I} - \mathcal{D}_{0} \right)^{-1} = - \sum_{p = 1}^{ \infty } \mathcal{D}_{0}^{p}.
\end{equation*}
Here, each term in the expansion is defined successively from the previous term, i.e., $\mathcal{D}_{0}^{p} = \mathcal{D}_0 [\mathcal{D}_{0}^{p-1}]$. Therefore, the action of $\partial_{xx}$ on a generic function $v(x)$ is given by
\begin{equation}
    \label{eq:second derivative of v}
    \partial_{xx}v(x) = -\alpha^{2} \sum_{p = 1}^{ \infty } \mathcal{D}_{0}^{p}[v; \alpha](x).
\end{equation}
As previous noted, expressions in multiple spatial dimensions can be obtained by simply changing labels, e.g., $x$ to $y$.  

\subsection{First order derivative $\partial_{x}$}
\label{sec:first order derivative}
As with the last section, we will use the same basic idea to construct an approximation to $\partial_{x}$ that will allow us to build an integral representation that provides an up wind and down wind approximation to our operator.

In order to obtain a representation for the first derivative, we introduce two operators: $\mathcal{L}_{L}$ and $\mathcal{L}_{R}$ to account for waves traveling in different directions. The subscript on an operator is used to identify the direction associated with wave propagation, so that ``$L$" and ``$R$" correspond to downwinding and upwinding, respectively. The operands for this decomposition would, of course, come from a monotone splitting, depending on the problem. With this convention, we define 
\begin{align}
\label{eq:left and right L operators}
    \mathcal{L}_{L} = \mathcal{I} - \frac{1}{\alpha}\partial_{x}, \quad
    \mathcal{L}_{R} = \mathcal{I} + \frac{1}{\alpha}\partial_{x},  \quad x \in [a,b],
\end{align} 
where $\mathcal{I}$ is the identity operator and, again,  $\alpha$ is a strictly positive constant. Using an integrating factor, we can invert these operators, similar to the case for $\mathcal{L}_{0}$ to find that
\begin{subequations}
    \begin{align}
        \label{eq:left inverse}
        & \mathcal{L}_{L}^{-1}[v,\alpha](x) = I_{L}[v,\alpha](x) + B_{L}e^{-\alpha (b - x)}, \\
        \label{eq:right inverse}
        & \mathcal{L}_{R}^{-1}[v,\alpha](x) = I_{R}[v,\alpha](x) + 
    A_{R} e^{-\alpha (x - a)},
    \end{align}
\end{subequations}
where
\begin{subequations}
    \begin{align}
        \label{eq:IL}
        & I_{L}[v,\alpha](x) = \alpha \int_{x}^{b} e^{-\alpha (s-x)}v(s) \, ds, \\
        \label{eq:IR}
        & I_{R}[v,\alpha](x) = \alpha \int_{a}^{x} e^{-\alpha (x-s)}v(s) \, ds,
    \end{align}
\end{subequations}
with constant $A_{R}$ and $B_{L}$ being determined by the boundary condition imposed for the operators.  These expressions depend on the problem, so handling the general case requires a substantial amount of care.   


As with the second derivative operator, we introduce the operators
\begin{equation}	
    \label{eq:operD}
    \mathcal{D}_{L} = \mathcal{I} - \mathcal{L}^{-1}_{L}, \quad
    \mathcal{D}_{R} = \mathcal{I} - \mathcal{L}^{-1}_{R}, \quad x\in[a,b].
\end{equation}
and expand each of these into a Neumann series:
\begin{subequations}
	\label{eq:sum}
	\begin{align}
    	\label{eq:LL}
    	& \frac{1}{\alpha}\partial_{x}^{+} = \mathcal{I}-\mathcal{L}_{L}
    	= \mathcal{L}_{L} (\mathcal{L}^{-1}_{L}-\mathcal{I})
    	= -\mathcal{D}_{L}/(\mathcal{I} -\mathcal{D}_{L} )
    	= -\sum_{p=1}^{\infty}\mathcal{D}_{L}^{p}, \\
    	\label{eq:LR}
    	& \frac{1}{\alpha}\partial_{x}^{-} = \mathcal{L}_{R}-\mathcal{I}
    	= \mathcal{L}_{R} (\mathcal{I}-\mathcal{L}_{R}^{-1})
    	= \mathcal{D}_{R}/(\mathcal{I} -\mathcal{D}_{R} )
    	= \sum_{p=1}^{\infty}\mathcal{D}_{R}^{p}.
	\end{align}
\end{subequations}
As before, these operators are defined successively, but we leave the operand at each $p$ as a generic function $v(x)$. Moreover, the $\pm$ signs on the expressions for the derivatives in \eqref{eq:LL} and \eqref{eq:LR} do not reflect the direction of propagation. Instead, they represent the direction of approach at an interface. For example, we use $\partial_{x}^{+}$ to indicate the right-sided approximation of the derivative, in $x$, along an interface.

\subsection{Periodic boundary conditions}
\label{sec:periodic}
In this section we show how to impose periodic boundary conditions for the line by line MOL$^T$ formulation we leverage in this work.

For problems with periodic boundary conditions, we make the requirement that
\begin{equation}
    \label{eq:periodic BC requirement}
    \mathcal{D}_{L}^{p}[v;\alpha](a) =
    \mathcal{D}_{L}^{p}[v;\alpha](b), \quad 
    \mathcal{D}_{R}^{p}[v;\alpha](a) = 
    \mathcal{D}_{R}^{p}[v;\alpha](b), \quad 
    \mathcal{D}_{0}^{p}[v;\alpha](a) = 
    \mathcal{D}_{0}^{p}[v;\alpha](b), \quad p \geq 1.
\end{equation}
Using the definition of these operators \eqref{eq:operD} and \eqref{eq:D0 operator}, the above condition shows that at each level, we should select
\begin{equation}
    \label{eq:A and B for periodic problems}
    A_{R} = \frac{ I_{R}[v;\alpha](b) }{1 - \mu}, \quad B_{L} = \frac{ I_{L}[v;\alpha](a) }{1 - \mu}, \quad A_{0} = \frac{ I_{0}[v;\alpha](b) }{1 - \mu}, \quad B_{0} = \frac{ I_{0}[v;\alpha](a) }{1 - \mu},
\end{equation}
where $\mu = e^{-\alpha( b - a)}$. Hence, following the idea in \cite{christlieb2017kernel}, when $\phi$ is a periodic function, we can approximate the first derivative $\phi^{\pm}_{x}$ with (modified) partial sums in \eqref{eq:sum}, 
\begin{subequations}	
\label{eq:partialsum_per}
\begin{align}
	\phi_{x}^{+}(x)\approx \mP^{L}_{k}[\phi,\alpha](x)  = \left\{\begin{array}{ll}
	-\alpha\sum\limits_{p=1}^{k}\mD_{L}^{p}[\phi,\alpha](x), & k=1, 2,\\
	-\alpha\sum\limits_{p=1}^{k}\mD_{L}^{p}[\phi,\alpha](x) + \alpha\mD_{0}*\mD_{L}^2[\phi,\alpha](x), & k=3.\\
	\end{array}
	\right.
\end{align}
\begin{align}
	\phi_{x}^{-}(x)\approx \mP^{R}_{k}[\phi,\alpha](x) = \left\{\begin{array}{ll}
	\alpha\sum\limits_{p=1}^{k}\mD_{R}^{p}[\phi,\alpha](x), & k=1, 2,\\
	\alpha\sum\limits_{p=1}^{k}\mD_{R}^{p}[\phi,\alpha](x) - \alpha\mD_{0}*\mD_{R}^2[\phi,\alpha](x), & k=3,\\
	\end{array}
	\right.
\end{align}
\end{subequations}
Note that there is an extra term for $k = 3$. As remarked in  \cite{christlieb2017kernel}, such a term is needed for 
attaining unconditional stability of the scheme. An error estimate for the approximation \eqref{eq:partialsum_per} regarding the truncation of the infinite sum,  carried out in \cite{christlieb2017kernel},  showed that keeping $k$ terms of the partial sums led to 
\begin{subequations}
	\begin{align*}
	& \|\partial_{x}\phi(x)-\mP^{L}_{k}[\phi,\alpha](x) \|_{\infty}\leq C \left(\frac{1}{\alpha}\right)^{k} \|\partial_{x}^{k+1}\phi(x)\|_{\infty},\\
	& \|\partial_{x}\phi(x)-\mP^{R}_{k}[\phi,\alpha](x) \|_{\infty}\leq C \left(\frac{1}{\alpha}\right)^{k} \|\partial_{x}^{k+1}\phi(x)\|_{\infty}.
	\end{align*}
\end{subequations}
for the representation of the first derivative and
\begin{equation*}
    \|\partial_{xx}\phi(x) + \alpha^{2} \sum_{ p = 1}^{k} \mD_{0}^{p}[\phi,\alpha](x) \|_{\infty} \leq C \left( \frac{1}{\alpha} \right)^{2k} 
    \|\partial_{x}^{2k+1}\phi(x)\|_{\infty}
\end{equation*}
for the second derivative

In numerical simulations, we will take $\alpha =\beta/(c\Delta t)$ in \eqref{eq:partialsum_per}, with $c$ being the maximum wave propagation speed. Here, $\Delta t$ denotes the time step and $\beta$ is a constant independent of $\Delta t$. Hence, the partial sums approximate $\phi_{x}$ with accuracy $\mathcal{O}(\Delta t^k)$.

\subsection{Non-periodic boundary conditions}
\label{sec:non_periodic}

In this subsection, we will focus on the application of non-periodic boundary conditions to \eqref{eq:HJ}. Specific details for the error analysis, as well as more generic boundary conditions, can be found in our previous work on the H-J equation \cite{christlieb2019kernel}. 

For non-periodic problems, additional requirements imposed on the operators $\mathcal{D}_{*}$ need to be consistent with the boundary condition specified on $\phi$. Otherwise, this can lead to order reduction in the method. Using integration by parts, one can identify the source of the order reduction, which involves evaluations of $\phi$ and its derivatives, along the boundaries. To address this issue, the partial sums, as presented above, were modified to annihilate terms which resulted in the order reduction. Before introducing the modified partial sums, we specify certain requirements on the coefficients $A_{*}$ and $B_{*}$ used in the construction of a given operator $\mD_{*}$. 

\subsubsection{Conditions for $A_{*}$ and $B_{*}$}

For reconstructions of first derivatives, suppose that $C_{a}$ and $C_{b}$ are given numbers. We will explain, later, how these values are obtained. If we require
\begin{equation*}
    \mD_{R}[v,\alpha](a) = C_{a}, \quad \mD_{L}[v,\alpha](b) = C_{b},
\end{equation*}
then one can use the definition \eqref{eq:operD} to show that we should select
\begin{equation}
    \label{eq:fx_bc_dir}
    A_{R} = v(a) - C_{a}, \quad B_{L} = v(b) - C_{b}.
\end{equation}
As an example, suppose we wish to use a first order approximation of the first partial derivative, i.e.,
\begin{equation}
    \label{eq:Approx of phi_x at boundaries}
    \phi_{x}^{+} \approx -\alpha \mD_{L}[\phi; \alpha](x), \quad \phi_{x}^{-} \approx \alpha \mD_{R}[\phi; \alpha](x).
\end{equation}
According to an analysis of the truncation error, we should select
\begin{equation*}
    C_{a} \approx \frac{1}{\alpha}\phi_{x}(a), \quad C_{b} \approx -\frac{1}{\alpha}\phi_{x}(b),
\end{equation*}
to obtain a convergent approximation. The derivatives can be constructed using finite differences of a suitable order. 

The case for the second derivative is a bit more cumbersome, however, it works in essentially the same way. Again, if require
\begin{equation*}
    \mD_{0}[v,\alpha](a) = C_{a}, \quad \mD_{0}[v,\alpha](b) = C_{b},
\end{equation*}
where $C_{a}$ and $C_{b}$ are chosen to obtain appropriate approximation order, i.e.,
\begin{equation*}
    C_{a} \approx -\frac{1}{\alpha^2}\phi_{xx}(a), \quad C_{b} \approx -\frac{1}{\alpha^2}\phi_{xx}(b),
\end{equation*}
then the coefficients $A_{0}$ and $B_{0}$ are given as
\begin{subequations}
	\label{eq:gxx_bc_dir}
	\begin{align}
	& A_{0}=\frac{1}{1-\mu^2}\left( \mu\left(I_{0}[v,\gamma](b)-v(b)+C_{b}\right) - \left(I_{0}[v,\gamma](a)-v(a)+C_{a}\right)\right), \\
	& B_{0}=\frac{1}{1-\mu^2}\left( \mu\left(I_{0}[v,\gamma](a)-v(a)+C_{a}\right) - \left(I_{0}[v,\gamma](b)-v(b)+C_{b}\right)\right).
	\end{align}
\end{subequations}

The process of determining $C_{*}$ becomes difficult to generalize if more terms in the partial sums are required. Instead, another modification was proposed in  \cite{christlieb2019kernel}: Rather than specify the conditions for $A_{*}$ and $B_{*}$, the  partial  sums  were  modified so that boundary-related terms, which led to order reduction, were automatically removed.

\subsubsection{The modified partial sums}
\label{sec:modify}

In developing the modified sums for the first derivative, we assume that the derivatives of $\phi$ have been constructed, in some way, at the boundaries, i.e., $\partial_{x}^{m}\phi(a)$ and $\partial_{x}^{m}\phi(b)$, $m\geq1$. Using this information, the schemes presented in  \cite{christlieb2019kernel}, which address non-periodic boundary conditions, were, as follows:
\begin{subequations}
	\label{eq:partialsum_dir}
	\begin{align}
	\phi_{x}^{-}(x)\approx\widetilde{\mP}^{L}_{k}[\phi,\alpha](x)=\left\{\begin{array}{ll}
	\alpha\sum\limits_{p=1}^{k}\mathcal{D}_{L}[\phi_{1,p},\alpha](x), & k=1,\, 2,\\
	\alpha\sum\limits_{p=1}^{k}\mathcal{D}_{L}[\phi_{1,p},\alpha](x) -\alpha \mD_{0}[\phi_{1,3},\alpha](x), & k=3,\\
	\end{array}
	\right.
	\end{align}
	\begin{align}
	\phi_{x}^{+}(x)\approx\widetilde{\mP}^{R}_{k}[\phi,\alpha](x)=\left\{\begin{array}{ll} -\alpha\sum\limits_{p=1}^{k}\mathcal{D}_{R}[\phi_{2,p},\alpha](x),& k=1,\, 2,\\
	-\alpha\sum\limits_{p=1}^{k}\mathcal{D}_{R}[\phi_{2,p},\alpha](x) +\alpha \mD_{0}[\phi_{2,3},\alpha](x), & k=3.\\
	\end{array}
	\right.
	\end{align}
\end{subequations}
And $\phi_{1,p}$ and $\phi_{2,p}$ are given as
\begin{subequations}
	\label{eq:expression}
	\begin{align}
	& \left\{\begin{array}{ll}
	\phi_{1,1}=\phi,\\
	\displaystyle \phi_{1,2}=\mathcal{D}_{L}[\phi_{1,1},\alpha] - \sum_{m=2}^{k}\left(-\frac{1}{\alpha}\right)^{m} \partial_{x}^{m}\phi(a) e^{-\alpha(x-a)},\\
	\displaystyle \phi_{1,3}=\mathcal{D}_{L}[\phi_{1,2},\alpha] + \sum_{m=2}^{k}(m-1)\left(-\frac{1}{\alpha}\right)^{m} \partial_{x}^{m}\phi(a) e^{-\alpha(x-a)},\\
	\end{array}
	\right.\\
	& \left\{\begin{array}{ll}
	\phi_{2,1}=\phi, \\
	\displaystyle \phi_{2,2}=\mathcal{D}_{R}[\phi_{2,1},\alpha] - \sum_{m=2}^{k}\left(\frac{1}{\alpha}\right)^{m} \partial_{x}^{m}\phi(b) e^{-\alpha(b-x)}, \\
	\displaystyle \phi_{2,3}=\mathcal{D}_{R}[\phi_{2,2},\alpha] + \sum_{m=2}^{k}(m-1)\left(\frac{1}{\alpha}\right)^{m} \partial_{x}^{m}\phi(b) e^{-\alpha(b-x)},\\
	\end{array}
	\right.
	\end{align}
\end{subequations}
with the boundary conditions for the operators
\begin{align*}
& \alpha\mathcal{D}_{L}[\phi_{1,1},\alpha](a)=\phi_{x}(a), \quad
\alpha\mathcal{D}_{R}[\phi_{2,1},\alpha](b)=-\phi_{x}(b),\\
& \alpha\mathcal{D}_{L}[\phi_{1,p},\alpha](a)=\alpha\mathcal{D}_{R}[\phi_{2,p},\alpha](b)= 0,
\quad \text{for} \ p\geq2,\\
& \alpha\mD_{0}[\phi_{*,3},\alpha](a) = \alpha\mD_{0}[\phi_{*,3},\alpha](b)=0, \quad \text{$*$ could be 1 or 2.}
\end{align*} 
The modified partial sum \eqref{eq:partialsum_dir} is constructed so that it agrees with the derivative values at the boundary, to preserve consistency with the boundary condition imposed on $\phi$. Furthermore, the authors provided the following theorem, which verifies the accuracy of these modified sums:
\begin{thm}
	Suppose $\phi\in\mathcal{C}^{k+1}[a,b]$, $k=1, \, 2,\, 3$. Then, the modified partial sums \eqref{eq:partialsum_dir} satisfy
	\begin{subequations}
		\begin{align}
		& \|\partial_{x}\phi(x)-\widetilde{\mP}^{L}_{k}[\phi,\alpha](x)\|_{\infty}\leq C \left(\frac{1}{\alpha}\right)^{k} \|\partial_{x}^{k+1}\phi(x)\|_{\infty},\\
		& \|\partial_{x}\phi(x)-\widetilde{\mP}^{R}_{k}[\phi,\alpha](x)\|_{\infty}\leq C \left(\frac{1}{\alpha}\right)^{k} \|\partial_{x}^{k+1}\phi(x)\|_{\infty}.
		\end{align}
	\end{subequations}
\end{thm}
Recalling that we defined $\alpha=\beta/(\alpha\Delta t)$ shows that the modified partial sums \eqref{eq:partialsum_dir} approximate $\phi_{x}$ with accuracy $\mathcal{O}(\Delta t^k)$.

\section{Extensions of the scheme to nonuniform grids}

In this section we describe the extension of the method to mapped grid.  Section \ref{sec:map} reviews the fact that the H-J equation under a coordinate transformation yields yet another H-J equation.  It is this fact allows us to develop systematic approach to solving the H-J equation on mapped and non-mapped grids.  In Section 3.2, we develop exponential WENO kernel based operators that we use in the MOL$^T$ approximation to the H-J equation on mapped grids.   In Section 3.3 we outline the MOL$^T$ algorithm on mapped grids formulation of our H-J solver.

\subsection{Problem description on the physical domain}\label{sec:map}
In the one-dimensional case, \eqref{eq:HJ} becomes
\begin{align}\label{eq:1D}
\phi_{t}+ H(\phi_{x})=0, \quad a\leq x \leq b,
\end{align}
with $\phi(x,0)=\phi^{0}(x)$.
Assume that the spatial domain is a closed interval $[a,b]$ and partitioned with $N+1$ points
\begin{align*}
a=x_{0}<x_{1}<\cdots<x_{N-1} <x_{N} =b,
\end{align*}
with $\Delta x_{i}=x_{i+1}-x_{i}$ for $i=0,\cdots,N-1$. These grids of the physical domain could be nonuniform. 
Let $\phi_{i}(t)$ denote the solution $\phi(x_{i},t)$ at mesh point $x_i$ for $i=0,\cdots,N$. We start with the transformation from the physical domain to the computational domain. Let $\xi$ be the uniformly distributed coordinates on our computational domain $[0,1]$:
$$ 0= \xi_{0} < \xi_{1} < \cdots < \xi_{N-1} < \xi_{N} = 1, $$
so that $\xi_{i}=i\Delta \xi$ with $\Delta \xi = 1/N$, and define a one-to-one coordinate transformation by $$ x = x(\xi) : [0,1] \rightarrow [a,b], $$ with $ x(\xi_i) = x_i $, $ x(0)=a $ and $ x(1) = b $. With this transformation, we can convert the one-dimensional H-J equation \eqref{eq:1D} to a new H-J equation
\begin{align}\label{eq:1Dxi}
\phi_{t}+ \tilde{H}(\phi_{\xi})=0, \quad 0 \leq \xi \leq 1,
\end{align}
where
\begin{align}\label{transf:H}
\tilde{H}(\phi_{\xi}):=H(\xi_x\phi_{\xi}).
\end{align}

The proposed numerical scheme on the transformed spatial domain is developed according to the semi-discrete equation
\begin{align}\label{eq:1Dscheme}
\frac{d}{dt}\phi_{i}(t)+\hat{\tilde H}(\phi_{\xi,i}^{-},\phi_{\xi,i}^{+}) =0,\quad i=0,\ldots,N,
\end{align}
where $\hat{\tilde H}$ is a numerical Hamiltonian which is a Lipschitz continuous monotone flux consistent with $\tilde H$, i.e.,
$$\hat{\tilde H}(u,u)=\tilde H(u).$$
Here $\phi_{\xi,i}^{-}$ and $\phi_{\xi,i}^{+}$ are the approximations to $\phi_{\xi}$ at $\xi_{i}$ obtained by left-biased and right-biased methods, respectively, to take into the account the direction of characteristics propagation of the H-J equation.
In this work, the local Lax-Friedrichs flux
\begin{align}\label{eq:LLF}
\hat{\tilde H}(u^{-},u^{+})=\tilde H(\frac{u^{-}+u^{+}}{2}) -\alpha_{\tilde{H}}(u^{-},u^{+})\frac{u^{+}-u^{-}}{2}
\end{align}
is used with $\alpha_{\tilde{H}}(u^{-},u^{+})=\max_u |\tilde H'(u)|$ where $u \in [\min(u^{-},u^{+}), \max(u^{-},u^{+})]$.

\begin{lem}\label{lem0}
	The numerical scheme for \eqref{eq:1Dscheme} with \eqref{eq:LLF} is conservative in terms of $\phi_x$.
\end{lem}
\begin{proof}
	The equation \eqref{eq:1Dscheme} can be discretized with $n$-th timestep $\Delta t$ by
	\begin{align}\label{numeq:1d}
	\phi_{i}^{n+1} = \phi_{i}^{n} - \Delta t \hat{\tilde H}(\phi_{\xi,i}^{-},\phi_{\xi,i}^{+})
	\end{align}
	where $\phi^{n}$ denotes the semi-discrete solution at $t^{n}$.
	Then it can be proved easily following from the fact that the scheme for \eqref{numeq:1d} approximates hyperbolic conservation laws. First, we define a function $\Phi(x,t)$, which satisfies
	\begin{align*}
	\Phi(x_i)=\frac{1}{\Delta x_i}\int_{x_i}^{x_{i+1}} \phi_x =\frac{\phi_{i+1}-\phi_i}{\Delta x_i}, \quad i=0,\cdots,N-1, 
	\end{align*}
	and consider the time evolution of this function.
	Let $\Phi_i^n$ be the value of $\Phi$ at $x_i$ at the $n$-th timestep $t^n$
	and so we obtain that
	\begin{align}\label{def:Phi}
	\begin{split}
		\frac{\Phi_{i}^{n+1}-\Phi_{i}^{n}}{\Delta t} 
	&=\frac{1}{\Delta x_i} \left[ \frac{\phi_{i+1}^{n+1}-\phi_{i+1}^{n}}{\Delta t} - \frac{\phi_{i}^{n+1}-\phi_i^n}{\Delta t} \right].
	\end{split}
	\end{align}
	Denoting the Jacobian of the coordinate transformation $J=x_{\xi}$, we can find the relation
$	\phi_{\xi,i}^{\pm}=J_i \phi_{x,i}^{\pm}$
	where $J_i = J|_{\xi_i}$,
	and using this relation with \eqref{transf:H} and \eqref{eq:LLF}, the equation \eqref{numeq:1d} is converted to
	\begin{align*}
	\phi_{i}^{n+1}
	&=\phi_{i}^{n} - \Delta t \hat {\tilde H}(J_i\phi^-_{x,i},J_i\phi^+_{x,i}) \\
	&=\phi_{i}^{n} - \Delta t \left[ H \left( \frac{\phi_{x,i}^{-} + \phi_{x,i}^{+}}{2} \right) -\alpha_{\tilde{H}} \frac{J_i\phi_{x,i}^{+}-J_i\phi_{x,i}^{-}}{2} \right]
	\end{align*}
	with $\alpha_{\tilde{H}}=  \alpha_{\tilde{H}}(\phi_{\xi,i}^{-},\phi_{\xi,i}^{+})=\max |\tilde H'(\phi_{\xi,i})|$.
	Using this relation, in addition to \eqref{def:Phi}, we obtain
	\begin{align*}
	\frac{\Phi_{i}^{n+1}-\Phi_{i}^{n}}{\Delta t} 
	= \frac{1}{\Delta x_i}\left[ -\left\lbrace H \left( \frac{\phi_{x,i+1}^{-} + \phi_{x,i+1}^{+}}{2} \right) - \alpha_{\tilde{H}} J_{i+1}\frac{\phi_{x,i+1}^{+}-\phi_{x,i+1}^{-}}{2} \right\rbrace + \left\lbrace
	 H \left( \frac{\phi_{x,i}^{-} + \phi_{x,i}^{+}}{2} \right) -\alpha_{\tilde{H}} J_i\frac{\phi_{x,i}^{+}-\phi_{x,i}^{-}}{2} \right\rbrace \right],
	\end{align*}
    which is an update equation of the form
	\begin{align*}
	\frac{\Phi_{i}^{n+1}-\Phi_{i}^{n}}{\Delta t}
	= -  \frac{1}{\Delta x_i}\left[ \hat{H}( \phi_{x,i+1}^{-} , \phi_{x,i+1}^{+}) - \hat{H}( \phi_{x,i}^{-}, \phi_{x,i}^{+})\right].
	\end{align*}
	We identify $$\hat{H}(u^-, u^+)= H \left( \frac{u^{-} + u^{+}}{2} \right) - { \alpha}_{H}(u^-, u^+) \frac{u^{+}-u^{-}}{2}, $$
	with the relation
	$${\alpha_{H}}(u^-, u^+)=\max_u | H'(u)|=\max_u | \tilde{H}'(J u)|=J \alpha_{\tilde{H}}$$
	as the monotone numerical Hamiltonian, which is consistent with $H$, i.e.,
	$$\hat{H}(u,u)=H(u).$$ In other words, this is a conservative approximation to the hyperbolic conservation law $$ \Phi_t+H(\Phi)_x = 0. $$
\end{proof}

\subsection{Space discretization with exponential based WENO schemes}\label{S-wenone}
In this subsection, we present the detailed spatial discretization for the operators $\mathcal{D}_{L}$ and $\mathcal{D}_{R}$.
We can obtain \eqref{eq:IL} and \eqref{eq:IR} via recurrence relations for the integral terms:
\begin{align}
    &I_{R}[v; \alpha](x_{i}) = e^{-\alpha \Delta x_{i-1}} I_{R}[v; \alpha](x_{i-1}) +  J_{R}[v; \alpha](x_{i}), \label{eq:right-sweep}\\
    &I_{L}[v; \alpha](x_{i}) = e^{-\alpha \Delta x_{i}} I_{L}[v; \alpha](x_{i+1}) + J_{L}[v; \alpha](x_{i}), \label{eq:left-sweep}
\end{align}
where the local integrals are defined by
\begin{align}
\label{eq:JR}
J_{R}[v; \alpha](x_{i}) &= \alpha \int_{x_{i-1}}^{x_{i}} e^{-\alpha(x_{i} - s)} v(s) \,ds\\
\label{eq:JL}
J_{L}[v; \alpha](x_{i}) &= \alpha \int_{x_{i}}^{x_{i+1}} e^{-\alpha(s-x_{i})} v(s) \,ds.
\end{align}
By calculating the convolution integrals with this recurrence relation, we obtain a summation method which has a complexity of $\mathcal{O}(N)$ instead of $\mathcal{O}(N^2)$. 

In order to compute $J_{L}[v; \alpha](x_{i})$ and  $J_{R}[v; \alpha](x_{i})$, we propose a high order exponential based approximation. The process to approximate $J_{L}[v; \alpha](x_{i})$ is simply mirror symmetric to that of $J_{R}[v; \alpha](x_{i})$ with respect to point $x_{i}$, so we will illustrate the process only for the term $J_{R}[v; \alpha](x_{i})$. In \cite{HKYY}, the authors introduced a sixth-order WENO scheme based on the exponential polynomial space, which we shall follow for approximating $J_{R}[v; \alpha](x_{i})$. To begin, we consider an interpolation stencil consisting of $k+1$ points, which contains $x_{i-1}$ and $x_{i}$:
\begin{equation*}
S(i) = \{ x_{i-r}, \cdots, x_{i-r+k} \},
\end{equation*}
to find a unique polynomial of degree at most $k$, denoted as $p(x)$, which interpolates $v(x)$ at the points in $S(i)$ so that
\begin{equation*}
J_{R}[v; \alpha](x_{i}) \approx \alpha \int_{x_{i-1}}^{x_{i}} e^{-\alpha(x_{i} - s)} p(s) \,ds.
\end{equation*}
In this paper, a six-point stencil $S:=S(i)=\{x_{i-3},\cdots, x_{i+2}\}$ is used and this stencil is subdivided into three substencils $S_0, \cdots, S_2$ defined by $S_r:=S_{r}(i) =\{ x_{i-3+r},\cdots, x_{i+r}\}$ for $r = 0, 1$ and $2$.
The corresponding stencil is shown in Figure \ref{Fig0}.

\begin{figure}
	\centering
	\includegraphics[width=0.5\textwidth]{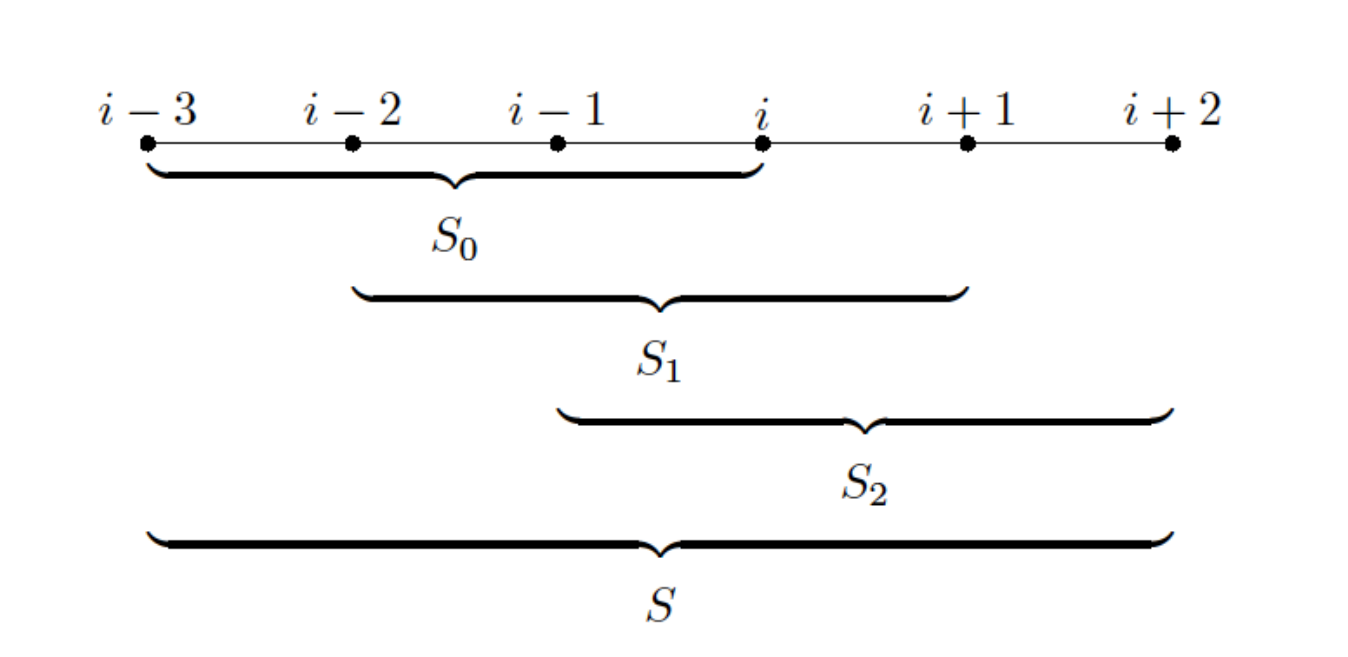}
	\caption{\em The $6$-point stencil $S$ with three $4$-point stencils $S_r$, $r=0,1$ and $2$.}
	\label{Fig0}
\end{figure}

Let $\{\phi_1,\cdots,\phi_k\}$ be a set of exponential polynomials of the form 
$\phi(x) =  x^n e^{\lambda x} $
with  $n\in\NN\cup \{0\}$ and a ``tension'' parameter {$\lambda$} which is used as chosen to improve the approximating ability according to the data. We can choose $\lambda\in\RR$ or $\lambda\in i\RR$ ($i= \sqrt{-1}$) and the function $\phi$ could be a trigonometric polynomial if $\lambda$ is in $i\RR$.
With these exponential polynomials as basis functions, we define a rank $k$ space $\Gamma_k$ by 
\begin{equation*}
    \Gamma_k :={\rm span} \{\phi_1 , \ldots, \phi_k \},
\end{equation*}
which satisfies
\begin{equation}
    \label{nonsing}
    \det (\phi_n(s_j): j, n = 1,\ldots, k )  \not= 0
\end{equation}
for a $k$-point stencil $\{s_j:j=1,\ldots,k\}$. It is recommended that the polynomial $\phi(x) \equiv 1$ is contained as a basis in order for an interpolation kernel to satisfy a partition of unity, so we choose
\begin{equation}
    \label{bigsp}
    \Gamma_6:={\rm span} \{1, x, x^2, x^3, e^{\lambda x}, e^{-\lambda x}\}
\end{equation}
as the basis functions for global stencil $S$ and similarly, 
\begin{equation}
    \label{subsp}
    \Gamma_4:={\rm span} \{1, x, e^{\lambda x}, e^{-\lambda x}\}
\end{equation}
for the four-point substencils.
Here $\Gamma_6$ and $\Gamma_4$ constitute {\em extended Tchebysheff system}s on $\RR$ so that the non-singularity of the interpolation matrices in \eqref{nonsing} is guaranteed, see \cite{KSbook,HKYY}.

On the big stencil $S(i)$, using the basis defined on $\Gamma_6$, we obtain the approximation as
	\begin{align}\label{eq:gloJR}
	\begin{split}
	J^{R}_{i} &:=\alpha \int_{x_{i-1}}^{x_{i}}e^{-\alpha(x_{i} - s)} p(s) \,ds,
	\end{split}
	\end{align}
	where $p$ is an interpolant for $v$ that satisfies
	$ J^{R}_{i} = 	J_{R}[v; \alpha](x_{i}) + \alpha \mO(\Delta x^6)$ if $v$ is smooth on $S(i)$. Similarly, on each of the smaller stencils, we have
	\begin{align}\label{eq:localJR}
	\begin{split}
	J^{R}_{i,r} :=\alpha \int_{x_{i-1}}^{x_{i}}e^{-\alpha(x_{i} - s)} p_r(s) \,ds,
	\end{split}
	\end{align}
	where $p_{r}$ is the interpolant to $v$ using the basis from $\Gamma_4$ on nodes $S_{r}(i)$ which satisfies $J^{R}_{i,r} = 	J_{R}[v; \alpha](x_{i}) + \alpha \mO(\Delta x^4) $ for smooth $v$. When the function $v$ is smooth, we would like to combine approximations on the smaller stencils $S_{r}(i)$ so they are consistent with those on the larger stencil $S(i)$, i.e.,
	\begin{equation}\label{weno:dr}
	J^{R}_{i} =\sum_{r=0}^{2}d_{r}J^{R}_{i,r}.	
	\end{equation}
	The coefficients $d_{r}$ are the linear weights which satisfy $\sum_{r=0}^{2}d_{r}=1$.
	
The construction of smoothness indicators is as follows:
	First, for $r=0,1,2$, we use $m$th-order generalized undivided differences ($m=2,3$) on $S_r$ defined by
	\begin{align}\label{weno:unDD}
	D^m_r v_{i} := \sum_{x_n \in S_{r}(i)}  c_{r,n}^{[m]} v(x_{n}),
	\end{align}
    which converge to $\Delta x^m v^{(m)}(x_{i})$ in higher convergence rate than classical undivided differences.
	Let $n_r$ indicate the number of points inside the stencil $S_r$ and define the coefficient vector ${{\bf c}}^{[m]}:= (c^{[m]}_{r,n}: x_n \in S_r)^T $ in \eqref{weno:unDD}
	by solving the linear system
	\begin{align*}
	{\bf V}\cdot {{\bf c}}^{[m]}= {\boldsymbol\delta}^{[m]},
	\end{align*}
	for the non-singular matrix
	\begin{align*}
	&{\bf V}:= \left(\frac{(x_n-x_i)^\ell}{\Delta x ^{\ell} \ell !}:x_n \in S_r, \ \ell=0,\ldots, n_r-1 \right)
	\end{align*}
	and ${\boldsymbol\delta}^{[m]}:= (\delta_{m,\ell} :\ell=0,\ldots, n_r-1)^T$. Then a simple calculation with Taylor expansion shows that
	\begin{align}\label{weno:unDD2}
	D^m_r v_{i} =\Delta x^m v^{(m)}(x_{i}) + O(\Delta x^{n_r}),
	\end{align}
	on the smooth region.
    We now define a measurement for the smoothness of data in each substencil by
	\begin{align}\label{weno:beta}
	\begin{split}
	& \beta_r := 
	\left| D^2_r v_{i} \right|^2 + \left| D^3_r v_{i}\right|^2,
	\qquad {\rm for}\quad r= 0, 1, 2,
	\end{split}
	\end{align}
	 and the global smoothness indicator $\tau$ is simply defined as the absolute difference between $\beta_{0}$ and $\beta_{2}$, i.e., $$\tau=|\beta_{0}-\beta_{2}|.$$
	
We form the final approximation using
	\begin{equation}\label{eq:wenofin}
	\mathcal{A}J^{R}_{i}:=\sum_{r=0}^{2}\omega_{r}J^{R}_{i,r}.
	\end{equation}
	The nonlinear weights $\omega_{r}$, in the above, are defined as
	\begin{equation}\label{weno:wr}
	\omega_{r}=\tilde{\omega}_{r}/\sum\limits_{s=0}^{2}\tilde{\omega}_{s} \quad \text{and} \quad	\tilde{\omega}_{r}=d_{r}\left( 1+\frac{\tau}{\epsilon+\beta_{r}} + \frac{1}{2}\left(\frac{\beta_r}{\epsilon+\tau}\right)^2 \right),
	\end{equation}
	where $\epsilon>0$ is a small number to avoid the denominator becoming zero. In our experiments we take $\epsilon=10^{-6}$. Note that we have employed the nonlinear weights based on the idea of the WENO-P+3 scheme proposed in \cite{wenop3}, which enjoys less dissipation and higher resolution compared with the classical WENO schemes.
	This construction enables the approximation for $J_{R}[v; \alpha](x_{i})$ to retain a high order of accuracy and it is proven in the following theorem. 

\begin{thm}
	Assume that $v$ is smooth on the global stencil $S(i)$. Then the approximation
	$\mathcal{A}J^{R}_{i}$ defined in \eqref{eq:wenofin} satisfies the relation
	\begin{align*}
	|J_{R}[v; \alpha](x_{i}) - \mathcal{A}J^{R}_{i}|=\alpha \mO(\Delta x^6),
	\end{align*}
	i.e., converges  to $J_{R}[v; \alpha](x_{i})$ in $5$th order accuracy if $\alpha \sim \frac{1}{\Delta x} $.
\end{thm}
\begin{proof}
	From the equations \eqref{weno:dr} and \eqref{eq:wenofin}, it is easily obtained that
	\begin{align*}
	J_{R}[v; \alpha](x_{i}) - \mathcal{A}J^{R}_{i}=
	\left(J_{R}[v; \alpha](x_{i})-\sum_{r=0}^{2}d_{r}J^{R}_{i,r}\right) + \sum_{r=0}^{2}(d_r - \omega_{r})J^{R}_{i,r}
	\end{align*}
	using the local approximations $J^R_{i,r}$ to $J_{R}[v; \alpha](x_{i})$.
	By definition of the linear weights $d_r$ in the \eqref{eq:gloJR} and \eqref{weno:dr}, the first term on the right-hand side has the designed order of accuracy so that it is sufficient to consider the second term. Since the local integral $J^R_{i,r}$ is constructed to have convergence in \eqref{eq:localJR} and both weights $\{d_r\}$ and $\{\omega_{r}\}$ fulfill a partition of unity, we have
	\begin{align*}
	\begin{split}
	\sum_{r=0}^{2}(d_r - \omega_{r})J^{R}_{i,r}&=\sum_{r=0}^{2}(d_r - \omega_{r})\left(J_{R}[v; \alpha](x_{i}) + \alpha \mO(\Delta x^4) \right)\\
	&= \alpha \sum_{r=0}^{2}(d_r - \omega_{r}) \mO(\Delta x^4)=\alpha \mO(\Delta x^6),
	\end{split}
	\end{align*}
	where the last equality is straightforward from \eqref{weno:unDD2} and \eqref{weno:wr} if we select $\epsilon=\Delta x^2$ in \eqref{weno:wr}.
\end{proof}

\begin{rem}
	In \cite{christlieb2017kernel, christlieb2019kernel}, the authors proposed to adapt nonlinear filters $\sigma_{i,L}$ and $\sigma_{i,R}$ to control oscillations which arise when the derivative of the solution to the H-J equation develops discontinuities. For example, in the periodic case, the approximation is given by 
	\begin{align}\label{eq:change_per}
		\begin{split}
		& \phi_{x,i}^{-}=\alpha\mathcal{D}_{L}[\phi,\alpha](x_{i}) + \alpha\sum_{p=2}^{k}\sigma_{i,L}\mathcal{D}_{L}^{p}[\phi,\alpha](x_{i}),\\
		& \phi_{x,i}^{+}=-\alpha\mathcal{D}_{R}[\phi,\alpha](x_{i}) - \alpha\sum_{p=2}^{k}\sigma_{i,R}\mathcal{D}_{R}^{p}[\phi,\alpha](x_{i}).
		\end{split}
	\end{align}
	The authors applied WENO quadrature to approximate the operators $\mathcal{D}_{L}$ and $\mathcal{D}_{R}$, but only for the first step, which corresponds to $p=1$ in \eqref{eq:change_per}. The filter was then adapted for the case $p\geq2$, so that a cheaper linear quadrature, defined on fixed stencils, can be used.
\end{rem}

We design a filter by defining parameters $\theta_{i}$ as
\begin{align}
\theta_{i}=\frac{\min_r (| D^1_r v_{i} | + | D^2_r v_{i}|)+\epsilon}{\max_r(| D^1_r v_{i} | + | D^2_r v_{i}|)+\epsilon}, \quad r=0,\cdots,3
\end{align}
where $D^k_r v_{i}$ are undivided differences of order $k$ on four three-point stencils around $v_i$, defined in \eqref{weno:unDD2}. Then, we adopt a nondecreasing map $\mu$ designed with cubic $B$-splines for $\theta_i$, in Figure \ref{Fig_filters} and the filter $\sigma_{i}$ is defined as 
\begin{align}\label{filter}
\sigma_{i}=\mu(2\cdot \theta_{i}^2). 
\end{align}

\begin{figure}[h]
	\centering
	\includegraphics[width=0.4\textwidth]{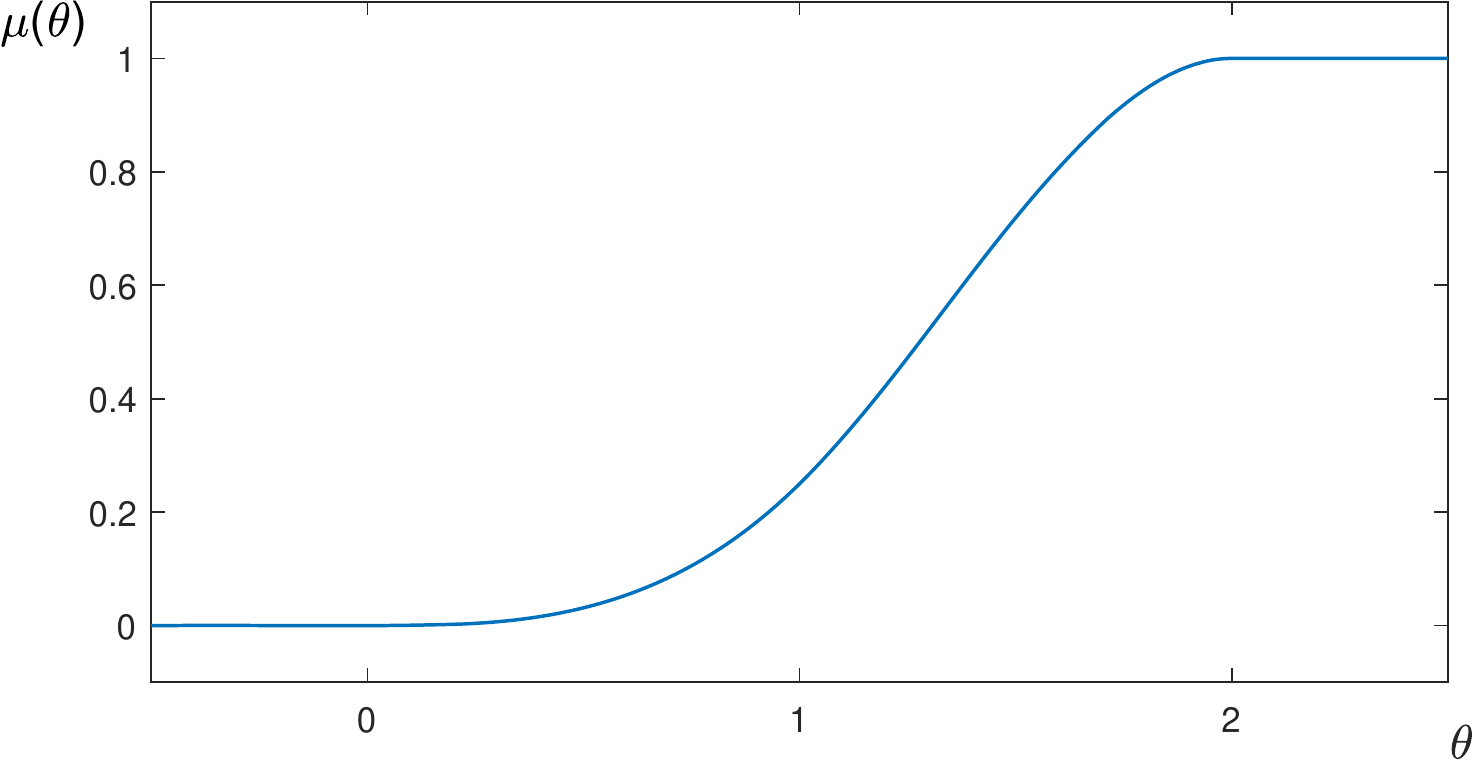}
	\caption{\em The function $\mu$ designed with cubic $B$-splines for filters $\sigma$.}
	\label{Fig_filters}
\end{figure}

\subsection{Algorithm}
In what follows, we will use $H$ and the coordinates $x$, rather than $\tilde{H}$ and coordinates $\xi$, for purposes of convenience. For time integration, we propose to use the classic explicit SSP RK methods \cite{gottlieb2001strong} to advance the solution from time $t^{n}$ to $t^{n+1}$. We denote $\phi^{n}$ as the semi-discrete solution at time $t^{n}$. In this work, we use the following SSP RK methods, including the first order forward Euler scheme
\begin{align}
    \label{eq:rk1}
    \phi^{n+1}=\phi^{n}-\Delta t \hat{H}(\phi^{n,-}_{x},\phi^{n,+}_{x});
\end{align} 
the second order SSP RK scheme
\begin{align}
    \label{eq:rk2}
    & \phi^{(1)}=\phi^{n}-\Delta t \hat{H}(\phi^{n,-}_{x},\phi^{n,+}_{x}),\nonumber\\
    & \phi^{n+1}=\frac{1}{2}\phi^{n}+\frac{1}{2}\left( \phi^{(1)} -\Delta t \hat{H}(\phi^{(1),-}_{x},\phi^{(1),+}_{x}) \right);
\end{align}
and the third order SSP RK scheme
\begin{align}
    \label{eq:rk3}
    & \phi^{(1)}=u^{n}-\Delta t \hat{H}(\phi^{n,-}_{x},\phi^{n,+}_{x}),\nonumber\\
    & \phi^{(2)}=\frac{3}{4}\phi^{n}+\frac{1}{4} \left( \phi^{(1)}-\Delta t \hat{H}(\phi^{(1),-}_{x},\phi^{(1),+}_{x}) \right), \nonumber\\
    & \phi^{n+1}=\frac{1}{3}\phi^{n}+\frac{2}{3} \left( \phi^{(2)}-\Delta t \hat{H}(\phi^{(2),-}_{x},\phi^{(2),+}_{x}) \right).
\end{align}
In addition, linear stability of the proposed kernel-based schemes has been established in \cite{christlieb2019kernel}:
\begin{thm}\label{thm-beta}
	For the linear equation $\phi_{t}+c\phi_{x}=0$, (i.e. the Hamiltonian is linear) with periodic boundary conditions, we consider the $k^{th}$ order SSP RK method as well as the  $k^{th}$ partial sum in \eqref{eq:partialsum_per}, with $\alpha=\beta/(|c|\Delta t)$. Then there exists a constant $\beta_{k,max}>0$ for $k=1,\,2,\,3$, such that the scheme is unconditionally stable provided  $0<\beta\leq\beta_{k,\max}$. The constants $\beta_{k,max}$ for $k=1,\,2,\,3$ are summarized in following:
	\begin{table}[htb]
		\centering
		\begin{tabular}{| l | p{1cm} | p{1cm} | p{1cm} |}
			\hline
			$k$ &  1  & 2  & 3  \\\hline
			$\beta_{k,max}$  &  2  &  1  &  1.243  \\\hline
		\end{tabular}
	\end{table}
\end{thm}

We summarize the proposed scheme for solving the one dimensional periodic boundary case in the following algorithm flowchart.

\def\HRule{\rule{\linewidth}{0.1mm}}
\medskip\noindent

{\sf
	\hskip -0.18truein\HRule
	
	\noindent
	\textbf{Algorithm}: MOL$^T$-type scheme for solving one-dimensional H-J equation on nonuniform grids
	\vskip -0.1truein
	\hskip -0.18truein\HRule
	
	\medskip\noindent
	We solve \eqref{eq:1D} until the final time $t=T$.
	Let the given nonuniform grid on the physical domain $[a,b]$ be $\{x_{i}\}$, $i=0,\ldots,N$.
	We denote the numerical solution at $n$-th time step $t=t^n$ by $\{\phi^n_{i}:=\phi^n(x_i)\}$. Start with $n=0$ and $\{\phi^0_i \}$ is given.
	
	\noindent
	\begin{enumerate}[label=\textbf{\arabic*.}]
		\setcounter{enumi}{0}
		\item 
		Define a computational grid $\{\xi_i\}$, $i=0,\ldots,N$, on $[0, 1]$ with uniform distribution $\Delta \xi = 1/N$.
		
		\item
		Approximate the associated Jacobian for the transformation $x_{\xi}$ by a fourth-order finite difference scheme:
		\begin{align}\label{jacobian}
			\{ x_{\xi} |_i = \frac{x_{i-2} - 8 x_{i-1} + 8x_{i+1} -x_{i+2} }{ 12 \Delta \xi}\},
		\end{align}
		and then we now solve the transformed H-J equation \eqref{eq:1Dxi}.
		
		\item 
		Set $\beta$ depending on desired order $k$ of the scheme according to Theorem \ref{thm-beta}
		and the time step size $\Delta {t^0}$ by CFL condition \eqref{cfl_1d}.
		
	\end{enumerate}
	
	\medskip\noindent
	While $t<T$, given $\phi^n$ and $\Delta t^n$ at time $t=t^n$, $n \geq 0$, 
	\begin{enumerate} [label=\textbf{\arabic*.}]
		\setcounter{enumi}{3}

		\item \label{algo:J}
		Approximate the integrals $J_L$ and $J_R$ using the exponential polynomial based WENO quadrature. 
		For $i=0,\ldots,N$,
		\begin{enumerate}
			\item 
			Construct the local approximations $J_{i,r}^L$ and $J_{i,r}^R$ based on exponential polynomials \eqref{eq:localJR}
			on each substencil $S_r$ for $r=0,1,2$,
			determine the linear weights $d_r$, and form the approximation given by equation \eqref{weno:dr}. 
			\item Compute the smoothness indicators of local data via \eqref{weno:beta} and construct nonlinear weights \eqref{weno:wr} for the final approximations $J_i^L$ and $J_i^R$ in \eqref{eq:wenofin}.

		\end{enumerate}

		\item  \label{algo:D} 
			Using the previously computed approximations $\mathcal{A}J_i^L$ and $\mathcal{A}J_i^R$, apply the recurrence relations \eqref{eq:right-sweep} and \eqref{eq:left-sweep} to obtain the convolution integrals. Form the inverse operators \eqref{eq:left inverse} and \eqref{eq:right inverse} by applying boundary conditions (see Section \ref{sec:periodic}). Once we have $\mathcal{L}^{-1}_{L}$ and $\mathcal{L}^{-1}_{R}$, construct the operators $\mathcal{D}_{L}$ and $\mathcal{D}_{R}$  defined in \eqref{eq:operD}.
			
		\item 
			Repeat step \ref{algo:J} and \ref{algo:D} $k$ times to approximate the first derivatives $\phi_{\xi}^-$ and $\phi_{\xi}^+$ with $k$ partial sums of $\mathcal{D}_{L}$ and $\mathcal{D}_{R}$, respectively. If multiple terms in the partial sums are desired, i.e., $k \geq 2$, then apply WENO quadrature only to the first terms and use the cheaper linear quadrature on the remaining ones. For this case, additional filters \eqref{filter}, obtained in WENO quadrature, are needed. Filters are applied according to \eqref{eq:partialsum_per}.
		
		\item 	
		Form the local Lax-Friedrichs Hamiltonian \eqref{eq:LLF} for the transformed H-J equation \eqref{eq:1Dscheme}
		using the inverse of associated Jacobian \eqref{jacobian}, {\it i.e.}, $\{{\xi}_x\}$. Then,
		update the time step from $t^n$ to $t^{n+1}=t^n+\Delta t^n$ by applying an appropriate RK scheme \eqref{eq:rk1} - \eqref{eq:rk3}. One should couple an order $k$ RK method to an approximation for the partial derivatives of an equivalent order for consistency. 
		
		\item 
		If $t^{n+1}<T$, set the time step size $\Delta t^{n+1}$ by \eqref{cfl_1d} with updated wave speeds. Otherwise, $t^{n+1}+\Delta{t}^{n+1}>T$, so we set $\Delta {t}^{n+1}=T-t^{n+1}$. Execute another time step of the process, beginning with \ref{algo:J}, until time $T$.
		
	\end{enumerate}
	
	\hskip -0.18truein\HRule
}

\subsection{Two-dimensional implementation}
\label{section:2D}

Consider the two-dimensional H-J equation
\begin{equation}\label{eq:2d}
\phi_{t}+H(\phi_{x},\phi_{y})=0,
\end{equation}
with $\phi(x,y,0)=\phi^{0}(x,y)$.
We assume $(x_{i}, y_{j})$ refers to the $(i,j)$-th node of a two-dimensional orthogonal grid. The spacing between points is denoted by $\Delta x_{i}=x_{i}-x_{i-1}$ and $\Delta y_{j}=y_{j}-y_{j-1}$. In addition, we shall take $\phi(x_{i},y_{j},t) = \phi_{i,j}(t)$ as the discrete solution to \eqref{eq:2d} on the grid.  

As with the one-dimensional case, we assume the existence of one-to-one coordinate transformations
$x=x(\xi, \eta)$ and $y = y(\xi, \eta)$ from the computational domain $[0,1]\times[0,1]$ to the physical domain. Here, the computational domain is distributed by a fixed uniform mesh given by $\xi_i=i\Delta \xi$ and  $\eta_j=j\Delta \eta$ with $\Delta \xi=1/N$ and $\Delta \eta=1/N$.
Then, the H-J equation \eqref{eq:2d} defined on irregular domain becomes
\begin{equation}
\phi_{t}+\tilde{H}(\phi_{\xi},\phi_{\eta})=0
\end{equation}
defined on uniform spatial domain, where $\tilde{H}(\phi_{\xi},\phi_{\eta}):=H(\xi_x\phi_{\xi}+\eta_x\phi_{\eta}, \xi_y\phi_{\xi}+\eta_y\phi_{\eta})$. Below, we will use $H$ and coordinates $x$ and $y$ instead of $\tilde{H}$ and coordinates $\xi$ and $\eta$. In the two-dimensional examples, we shall use the semi-discrete scheme 
\begin{align}
    \label{eq:2Dscheme}
    \frac{d}{dt}\phi_{i,j}(t) + \hat{H}( \phi^{-}_{x,i,j}, \phi^{+}_{x,i,j}; \phi^{-}_{y,i,j}, \phi^{+}_{y,i,j} ) = 0,
\end{align}
where the numerical Hamiltonian $\hat{H}(u,u;v,v)$ is a Lipschitz continuous monotone flux that is consistent with $H$. As in the one-dimensional case, we employ the local Lax-Friedrichs numerical Hamiltonian.
\section{Numerical results}\label{sec:result}
In this section, we present the numerical results of the proposed scheme for one-dimensional and two-dimensional Hamilton-Jacobi equations using regular and irregular grids discussed in Section \ref{subsec:1d}, \ref{subsec:2d} and \ref{subsec:map}, respectively.
The code implementation in Python with some sample results are available on the web \cite{yang_2019}.
The parameter $\lambda>0$ in the exponential basis can be tuned according to the problem, but in this paper, it is selected so that $\lambda \Delta \xi=1$ in all experiments. While there are examples with a  $\text{CFL} > 1$, unless otherwise stated, the presented numerical results are computed by the third-order scheme (i.e., $k=3$) using $\text{CFL} = 0.5$ to demonstrate the performance.

\subsection{One-dimensional cases}\label{subsec:1d}
Here, we present convergence results for the schemes on a one-dimensional uniform mesh with $\Delta x=\Delta x_i$ for $i=0,\cdots,N$. The time step is set by 
\begin{align}\label{cfl_1d}
    \Delta t=\text{CFL}\frac{\Delta x}{\alpha},
\end{align}
where $\alpha$ is the maximum wave propagation speed. We see the order of accuracy of proposed scheme for the linear and non-linear problems and present numerical results for several H-J examples.

\begin{exa} \label{ex1d:1}
We first solve the linear advection equation 
\begin{align}
    \phi_{t}+\phi_{x}=0 
\end{align}
on the spatial domain $[0,1]$ with periodic boundary conditions. For the initial condition, we use the smooth function 
\begin{align*}
    \phi(x,0)=\phi_1(x):=\sin(2\pi x).
\end{align*}
In Table \ref{tab:ex1}, we provide the $L_{\infty}$ errors at time $T=1$ and along with the associated order of accuracy. We can see that $k$th order of accuracy is achieved for $k=2$ and $3$ cases and second order accuracy is observed for the case $k=1$. Such superconvergence for the first order scheme, with $k=1$, is expected by observing that the proposed scheme, with $\beta=2$, applied to the linear problem, is equivalent to the second order Crank-Nicolson scheme \cite{christlieb2016weno}. 

\begin{table}[b!]
	\centering
\vspace{2mm}
	\begin{small}
		\begin{tabular}{|c|c|cc|cc|cc|}
			\hline
			\multirow{2}{*}{CFL} &  \multirow{2}{*}{$N$} & \multicolumn{2}{c|}{$k=1$. $\beta=2$.} & \multicolumn{2}{c|}{$k=2$. $\beta=1$.} & \multicolumn{2}{c|}{$k=3$. $\beta=1.2$.}\\
			\cline{3-8}
			& &  error &   order  &  error &  order  &  error  & order  \\\hline
			\multirow{5}{*}{0.5}  
			&  20  &  1.28e-02 &   --   &  1.73e-01  &   --     &  2.25e-03 &  -- \\
			&  40  &  3.23e-03 & 1.987  &  4.48e-02  &  1.947   &  1.72e-04 & 3.707  \\
			&  80  &  8.07e-04 & 1.998  &  1.13e-02  &  1.990   &  1.74e-05 & 3.310  \\
			& 160  &  2.02e-04 & 2.000  &  2.82e-03  &  1.998   &  2.02e-06 & 3.104  \\
			& 320  &  5.05e-05 & 2.000  &  7.06e-04  &  2.000   &  2.40e-07 & 3.076  \\\hline    
			\multirow{5}{*}{1} 
			&  20  &  5.08e-02 &   --   &  5.61e-01 &   --  &  3.15e-02  &  --      \\
			&  40  &  1.29e-02 & 1.981  &  1.73e-01 & 1.698 &  2.35e-03  & 3.744  \\
			&  80  &  3.23e-03 & 1.996  &  4.48e-02 & 1.948 &  1.86e-04  & 3.660  \\
			& 160  &  8.07e-04 & 1.999  &  1.13e-02 & 1.990 &  1.80e-05  & 3.368  \\
			& 320  &  2.02e-04 & 2.000  &  2.82e-03 & 1.998 &  2.06e-06  & 3.129  \\\hline      
			\multirow{5}{*}{2}  
			&  20  &  1.94e-01 &   --  &  9.92e-01 &   --   &  3.08e-01 &   --     \\
			&  40  &  5.09e-02 & 1.931 &  5.66e-01 & 0.810  &  3.16e-02 & 3.283 \\
			&  80  &  1.29e-02 & 1.984 &  1.73e-01 & 1.710  &  2.36e-03 & 3.747 \\
			& 160  &  3.23e-03 & 1.996 &  4.48e-02 & 1.948  &  1.86e-04 & 3.661 \\
			& 320  &  8.07e-04 & 1.999 &  1.13e-02 & 1.990  &  1.80e-05 & 3.370 \\\hline 
		\end{tabular}
	\end{small}	
	\caption{\label{tab:ex1} $L_{\infty}$-errors and orders of accuracy for  Example \ref{ex1d:1} with $\phi_1(x)$ at $T=1$. }
\end{table}

In the second case, we use the following initial function
\begin{equation*} 
\phi(x,0)=\phi_2(x) :=  \begin{cases}
0 & \text{if $ 0 \leq x \leq 0.25 $}, \\
\frac{40}{3}(x-\frac{1}{4}) & \text{if $ 0.25 < x < 0.4 $}, \\
2 & \text{if $ 0.4 \leq x \leq 0.6 $}, \\
\frac{40}{3}(\frac{3}{4}-x) & \text{if $ 0.6 < x < 0.75 $}, \\
0 & \text{if $ 0.75 \leq x \leq 1 $}
\end{cases}
\end{equation*}
which is a continuous and piecewise linear function. We plot the numerical solution and its derivative at time $T=1$, using $N=80$ grid points, in Figure \ref{fig:ex1}. We see that the proposed scheme improves the accuracy of the approximation near the jump discontinuity in the derivatives when compared to our previous scheme.

\begin{figure}[b!]
	\centering
	\vspace{2mm}
	\includegraphics[width=0.75\textwidth]{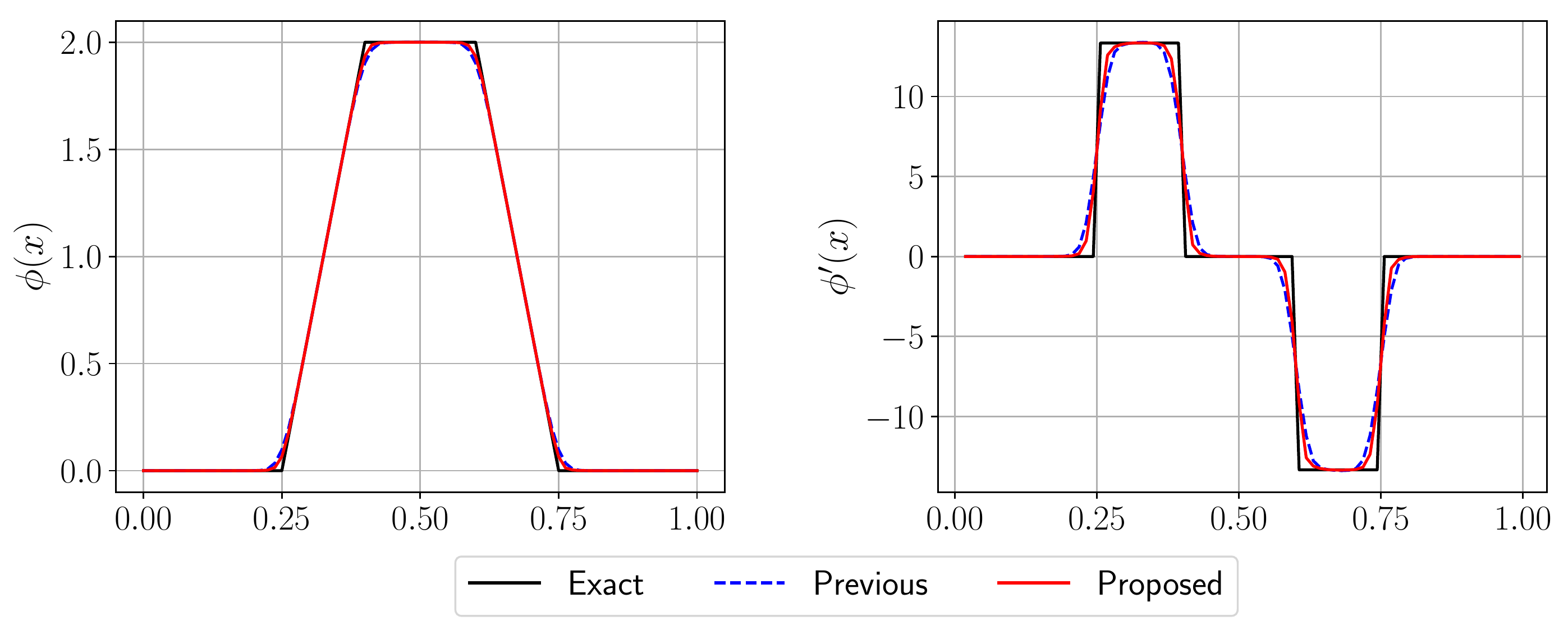}
	\caption{\label{fig:ex1} \em Numerical solution $\phi$ and its derivative $\phi_x$ for Example \ref{ex1d:1} with $\phi_2(x)$ at $T=1$.  }
\end{figure}

\end{exa}

\begin{exa}\label{ex1d:2}
\begin{figure}[b!]
	\centering
\vspace{3mm}
	\includegraphics[width=0.75\textwidth]{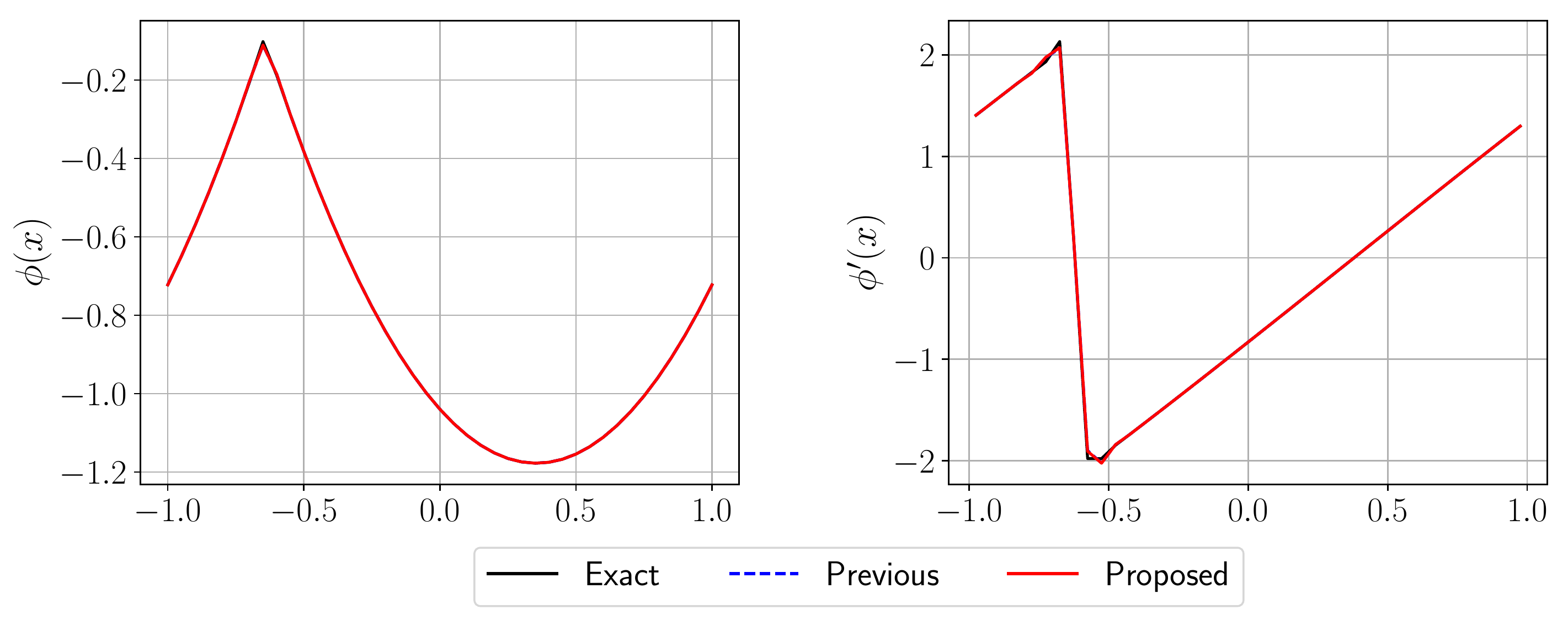}
	\caption{\em Numerical solution $\phi$ and its derivative $\phi_x$ for Example \ref{ex1d:2} at $T=3.5/\pi^2$. }
	\label{fig:ex2}
\end{figure}

In this example, we consider the Burgers' equation
\begin{align}
\phi_{t}+\frac{1}{2}(\phi_{x}+1)^2=0, 
\end{align}
with the smooth initial function
\begin{equation*}
\phi(x,0)=-\cos(\pi x)
\end{equation*}
on the spatial domain $[-1,1]$ with periodic boundary conditions at two different final times.

Our first test for this problem stops at time $T= 0.5/\pi^2$, when the solution is still smooth. In Table \ref{tab:ex2}, we provide the $L_{\infty}$ errors and orders of accuracy for schemes with $k=1,2$ and $3$, which are shown to be $k$-th order. Provided in Figure \ref{fig:ex2} is a plot of the numerical solution at $T= 3.5/\pi^2$, when the shock occurs. Here, $N=40$ grid points are used to compute the solution. It seems that the proposed scheme is as effective as the previous result.

\begin{table}[t!]
	\centering
	\begin{small}
		\begin{tabular}{|c|c|cc|cc|cc|}
			\hline
			\multirow{2}{*}{CFL} &  \multirow{2}{*}{$N$} & \multicolumn{2}{c|}{$k=1$. $\beta=2$.} & \multicolumn{2}{c|}{$k=2$. $\beta=1$.} & \multicolumn{2}{c|}{$k=3$. $\beta=1.2$.}\\
			\cline{3-8}
			& &  error &   order  &  error &  order  &  error  & order  \\\hline
			\multirow{5}{*}{0.5}  
			&  20  &  1.57e-02 & --     &  1.69e-02 &  --   &   3.30e-04 & --     \\
			&  40  &  7.94e-03 & 0.985  &  4.73e-03 & 1.839 &   3.63e-05 &  3.186  \\
			&  80  &  4.09e-03 & 0.958  &  1.29e-03 & 1.881 &   4.58e-06 &  2.986  \\
			& 160  &  2.07e-03 & 0.980  &  3.31e-04 & 1.958 &   5.20e-07 &  3.137  \\
			& 320  &  1.05e-03 & 0.987  &  8.47e-05 & 1.966 &   6.00e-08 &  3.115  \\\hline            
			\multirow{5}{*}{1}  
			&  20  &  3.19e-02 &  --    &  5.48e-02 &  --    &  4.06e-03 & --     \\
			&  40  &  1.57e-02 & 1.025  &  1.70e-02 & 1.691  &  5.08e-04 & 2.998  \\
			&  80  &  7.95e-03 & 0.980  &  4.76e-03 & 1.833  &  4.98e-05 & 3.351  \\
			& 160  &  4.12e-03 & 0.947  &  1.29e-03 & 1.890  &  5.00e-06 & 3.316  \\
			& 320  &  2.07e-03 & 0.992  &  3.31e-04 & 1.959  &  5.32e-07 & 3.234  \\\hline 
			\multirow{5}{*}{2}  
			&  20  &  9.14e-02 &   --   &  1.92e-01 & --     &  3.36e-02 & --     \\
			&  40  &  3.23e-02 & 1.501  &  5.67e-02 & 1.757  &  4.51e-03 & 2.898  \\
			&  80  &  1.57e-02 & 1.042  &  1.71e-02 & 1.731  &  5.12e-04 & 3.139  \\
			& 160  &  8.03e-03 & 0.966  &  4.76e-03 & 1.842  &  5.00e-05 & 3.356  \\
			& 320  &  4.12e-03 & 0.961  &  1.29e-03 & 1.890  &  5.01e-06 & 3.320  \\\hline 		
		\end{tabular}
	\end{small}	
	\caption{\label{tab:ex2}  $L_{\infty}$-errors and orders of accuracy for Example \ref{ex1d:2} at $T=0.5/\pi^2$. }
\vspace{2mm}
\end{table}

\end{exa}


\begin{exa}\label{ex1d:3}
We now solve the one-dimensional Riemann problem with a non-convex Hamiltonian
\begin{align}
\phi_{t}+\frac{1}{4} \left(\phi_{x}^2-1\right) \left( \phi_{x}^2-4 \right) = 0. 
\end{align}
Here, we use the initial condition
\begin{equation*}
\phi(x,0)=-2|x|,
\end{equation*}
on the fixed spatial domain $[-1,1]$ with the inflow Dirichlet boundary conditions $\phi(\pm 1,t)=-2$. 
In Figure \ref{fig1d:3}, we show plots of the numerical solution, which is computed up to time $T = 1$ using $N=80$ grid points. We measure convergence relative to a reference solution that is computed with $N=1600$ grid points. Both the previous and proposed methods are effective at resolving the reference solution.

\begin{figure}[h]
	\vspace{2mm}
	\centering	
	\includegraphics[width=0.75\textwidth]{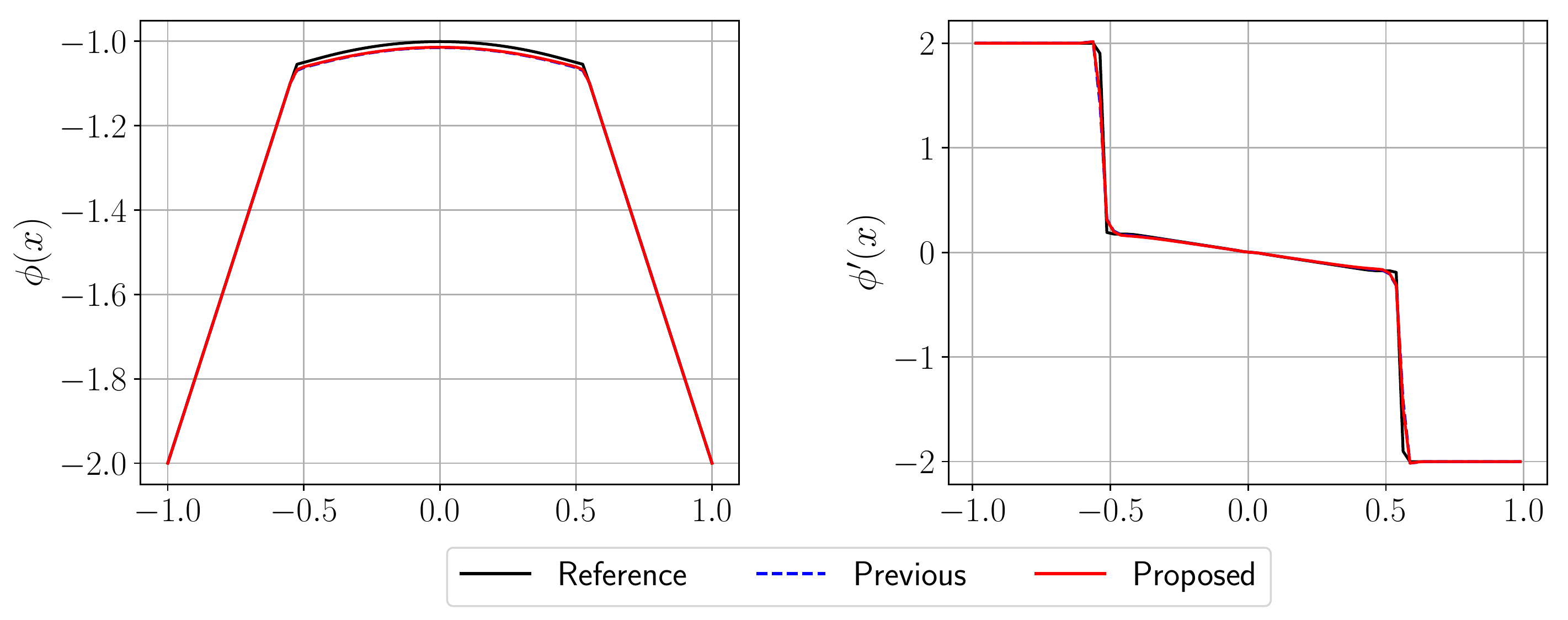}
	\caption{\em Numerical solution $\phi$ and its derivative $\phi_x$ for Example \ref{ex1d:3} at $T=1$. }
	\label{fig1d:3}
\end{figure}

\end{exa}


\subsection{Two-dimensional cases}\label{subsec:2d}

For the two-dimensional cases with uniformly distributed grids, the time step is chosen as
\begin{align}
\Delta t=\frac{\text{CFL}}{\max(\alpha_{x}/\Delta x,\alpha_{y}/\Delta y)},
\end{align}
where $\alpha_x$ and $\alpha_y$ are the maximum wave propagation speeds in the $x$ and $y$ directions, respectively. To demonstrate the performance of proposed scheme, we present the numerical solutions using $\text{CFL} = 0.5$ as well as $\text{CFL} = 2$.

\begin{exa}\label{ex2d:1}
	    \begin{table}[b!]
    	\centering
    	\vspace{2mm}
    	\begin{small}
    		\begin{tabular}{|c|c|cc|cc|cc|}
    			\hline
    			\multirow{2}{*}{CFL} &  \multirow{2}{*}{$N_x\times N_y$} & \multicolumn{2}{c|}{$k=1$. $\beta=1$.} & \multicolumn{2}{c|}{$k=2$. $\beta=0.5$.} & \multicolumn{2}{c|}{$k=3$. $\beta=0.6$.}\\
    			\cline{3-8}
    			& &  error &   order  &  error &  order  &  error  & order  \\\hline	
    			\multirow{5}{*}{0.5}  
    			&   $20\times20$  &  1.28e-02 & --     &  1.73e-01 & --     &  2.25e-03 & --   \\
    			&   $40\times40$  &  3.22e-03 & 1.987  &  4.48e-02 & 1.947  &  1.69e-04 & 3.736   \\
    			&   $80\times80$  &  8.07e-04 & 1.998  &  1.13e-02 & 1.990  &  1.65e-05 & 3.356   \\
    			& $160\times160$  &  2.02e-04 & 2.000  &  2.82e-03 & 1.998  &  1.90e-06 & 3.115   \\
    			& $320\times320$  &  5.05e-05 & 2.000  &  7.06e-04 & 2.000  &  2.33e-07 & 3.032   \\\hline                   	
    			\multirow{5}{*}{1} 
    			&   $20\times20$  & 5.08e-02 & --     &  5.61e-01 & --    &  3.15e-02 & --   \\
    			&   $40\times40$  & 1.29e-02 & 1.981  &  1.73e-01 & 1.698 &  2.35e-03 & 3.744\\
    			&   $80\times80$  & 3.23e-03 & 1.996  &  4.48e-02 & 1.948 &  1.86e-04 & 3.662\\
    			& $160\times160$  & 8.07e-04 & 1.999  &  1.13e-02 & 1.990 &  1.80e-05 & 3.370\\
    			& $320\times320$  & 2.02e-04 & 2.000  &  2.82e-03 & 1.998 &  2.04e-06 & 3.135\\\hline 
    			\multirow{5}{*}{2}  
    			&   $20\times20$  &  1.94e-01 & --    &  9.92e-01 & --    & 3.08e-01 & --   \\
    			&   $40\times40$  &  5.09e-02 & 1.931 &  5.66e-01 & 0.810 & 3.16e-02 & 3.283 \\
    			&   $80\times80$  &  1.29e-02 & 1.984 &  1.73e-01 & 1.710 & 2.36e-03 & 3.747 \\
    			& $160\times160$  &  3.23e-03 & 1.996 &  4.48e-02 & 1.948 & 1.86e-04 & 3.661 \\
    			& $320\times320$  &  8.07e-04 & 1.999 &  1.13e-02 & 1.990 & 1.80e-05 & 3.370 \\\hline   
    		\end{tabular}
    	\end{small}
    	\caption{\label{tab2d:1} 
    		$L_{\infty}$-errors and orders of accuracy for  Example \ref{ex2d:1} with $\phi_3(x,y)$ at $T=2$.
    	}
    \end{table}	
    	
    \begin{figure}[b!]
    	\centering	
    	\vspace{2mm}
    	\subfigure[Previous]{\includegraphics[width=0.43\textwidth]{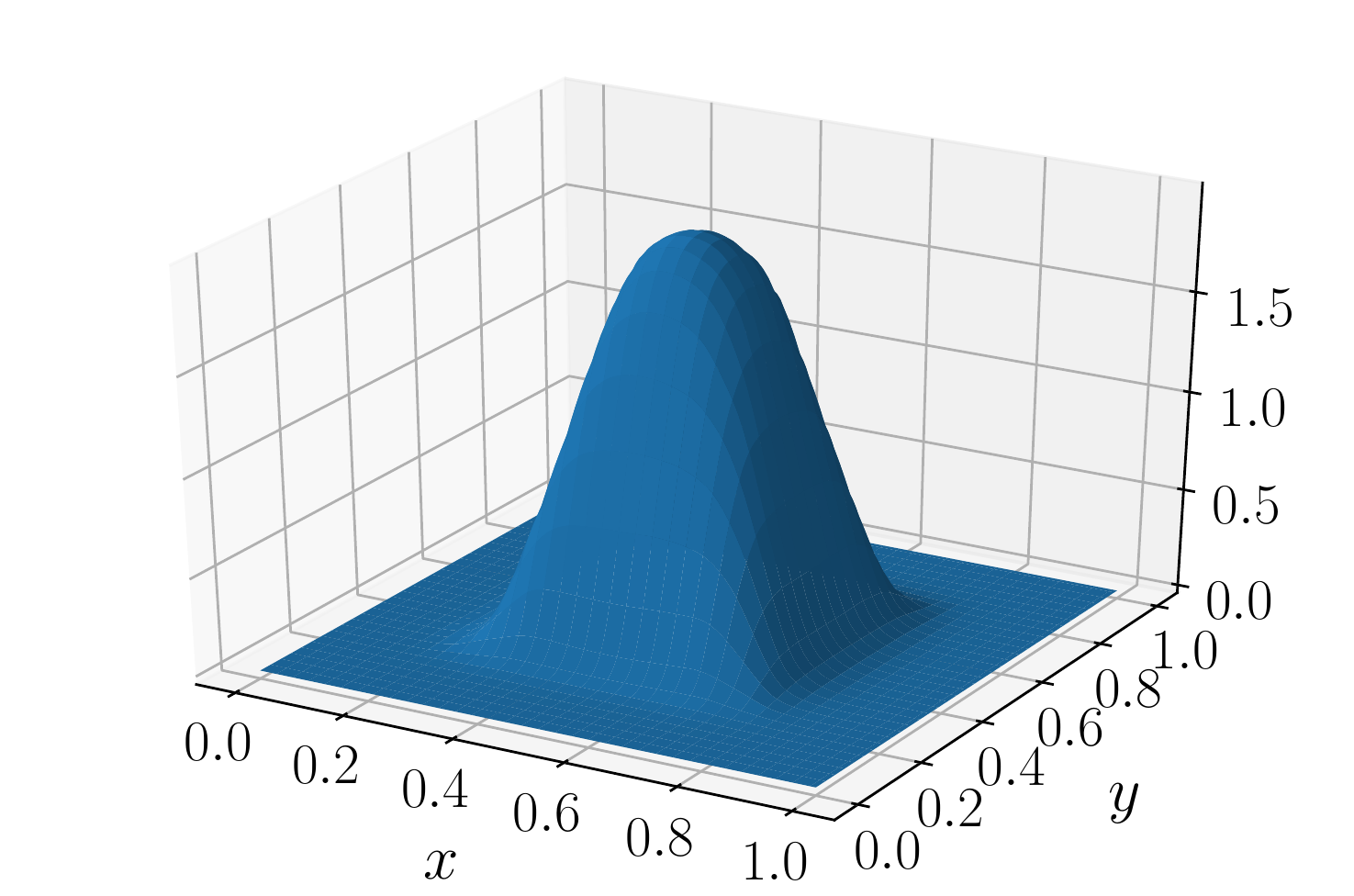}}
    	\subfigure[Proposed]{\includegraphics[width=0.43\textwidth]{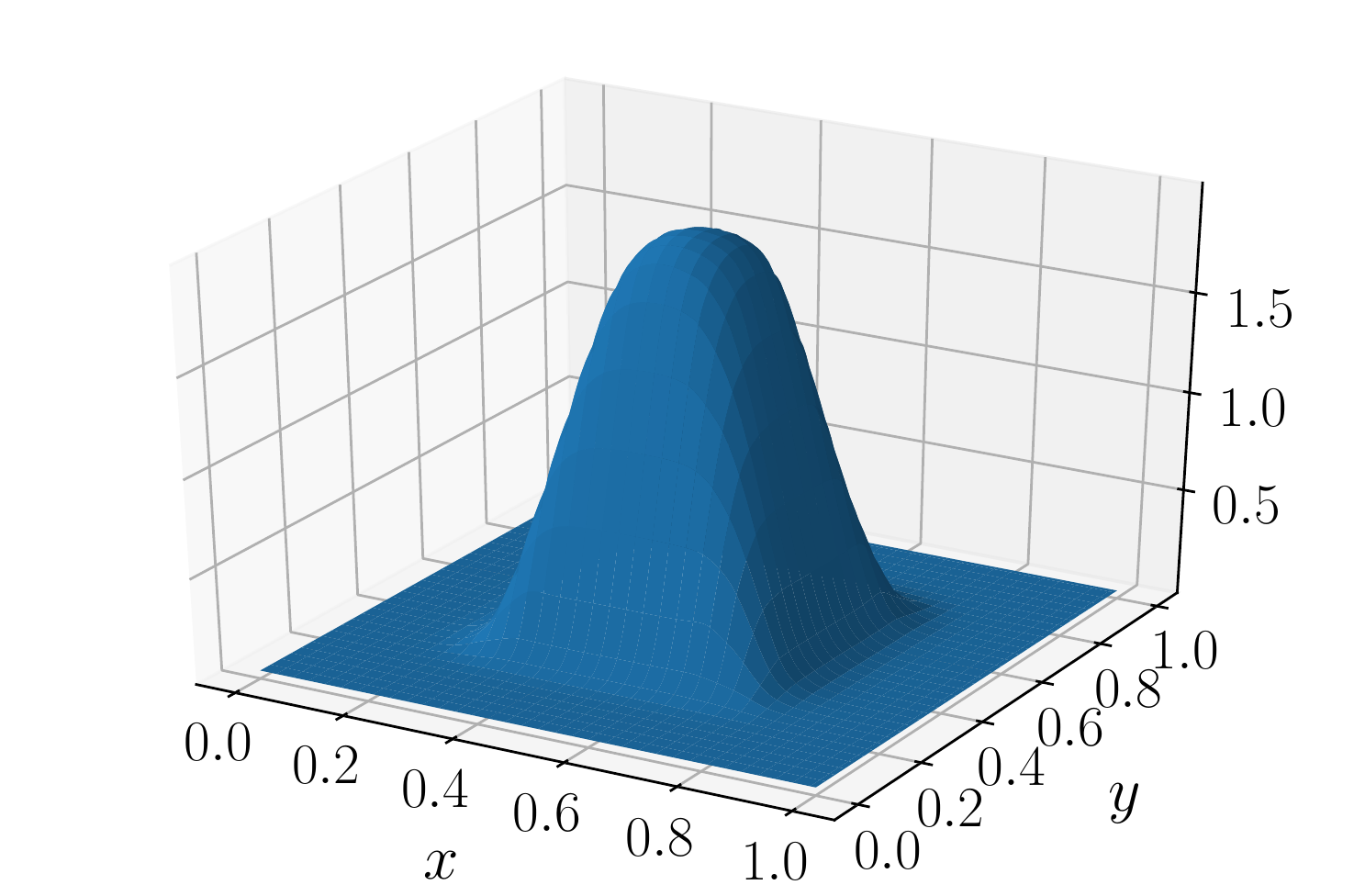}}
    	\caption{\em Numerical solutions for Example \ref{ex2d:1} with $\phi_4(x,y)$ at $T=1$.}
    	\label{fig2d:1}
    \end{figure}
    
    For our first example, we consider the linear advection equation
	\begin{align}
    	\phi_{t}+(\phi_{x}+\phi_{y}+1)=0, 
	\end{align}
	and we apply a periodic boundary condition in each direction. We first consider the smooth initial data given by
	\begin{equation*}
    	\phi(x,y,0) = \phi_3(x,y):=-\cos(\pi(x+y)/2),
	\end{equation*}
	which is defined on the spatial domain $[-2,2]\times[-2,2]$. We report the $L_\infty$ error and orders of accuracy at time $T=2$ in Table \ref{tab2d:1}. It is observed that the proposed scheme achieves the appropriate convergence orders as it does for the one-dimensional case in Example \ref{ex1d:1}. As before, we have the superconvergence for $k=1$ case. 
	
	In the second test problem, we test our method on a piece-wise continuous, i.e., $C^0$, function 
	\begin{equation*} 
	\phi(x,y,0)=\phi_4(x,y) :=  \begin{cases}
	0 & \text{ if  $x \leq 0.2$ { } or { } $y \leq 0.2$}, \\ 
	2 & \text{ if {  } $ 0.4 \leq x,y \leq 0.6 $}, \\
	0 & \text{ if  $x\geq 0.8$ { } or { } $y\geq 0.8$}, \\
	f(x,y) & \text{otherwise}
	\end{cases}
	\end{equation*}
	defined on the domain $[0,1]\times[0,1]$, where $f(x,y)$ is a real-valued piece-wise linear function defined so that $\phi_4$ is continuous. We display the result obtained with the proposed scheme, as well as the previous approach using a $100\times100$ grid of points. Plots of the numerical solutions at the time $T=1$ are provided in Figure \ref{fig2d:1}. As with Example \ref{ex1d:1}, we can see the resolution improvements of the proposed scheme around sharp edges, compared with the previous scheme.

\end{exa}

\begin{exa}\label{ex2d:2}

\begin{table}[b!]
	\centering
	\vspace{2mm}
	\begin{small}
		\begin{tabular}{|c|c|cc|cc|cc|}
			\hline
			\multirow{2}{*}{CFL} &  \multirow{2}{*}{$N_x\times N_y$} & \multicolumn{2}{c|}{$k=1$. $\beta=1$.} & \multicolumn{2}{c|}{$k=2$. $\beta=0.5$.} & \multicolumn{2}{c|}{$k=3$. $\beta=0.6$.}\\
			\cline{3-8}
			& &  error &   order  &  error &  order  &  error  & order  \\\hline	
			\multirow{5}{*}{0.5}  
            &   $20\times20$  &  5.48e-02 & --    &  1.68e-02 & --     &  6.36e-04 & --    \\
            &   $40\times40$  &  2.98e-02 & 0.877 &  4.72e-03 & 1.836  &  5.11e-05 & 3.637 \\
            &   $80\times80$  &  1.63e-02 & 0.874 &  1.28e-03 & 1.880  &  4.29e-06 & 3.574 \\
            & $160\times160$  &  8.37e-03 & 0.961 &  3.30e-04 & 1.956  &  5.13e-07 & 3.065 \\
            & $320\times320$  &  4.27e-03 & 0.970 &  8.46e-05 & 1.965  &  6.03e-08 & 3.091 \\\hline                   	
            \multirow{5}{*}{1} 
            &   $20\times20$  &  9.84e-02 & --     &  5.89e-02 & --    &  7.60e-03 & --    \\
            &   $40\times40$  &  5.54e-02 & 0.829  &  1.69e-02 & 1.801 &  9.22e-04 & 3.043  \\
            &   $80\times80$  &  3.05e-02 & 0.860  &  4.74e-03 & 1.835 &  6.16e-05 & 3.904  \\
            & $160\times160$  &  1.63e-02 & 0.906  &  1.28e-03 & 1.887 &  4.75e-06 & 3.697  \\
            & $320\times320$  &  8.37e-03 & 0.961  &  3.30e-04 & 1.956 &  5.25e-07 & 3.177  \\\hline 
            \multirow{5}{*}{2}  
            &   $20\times20$  &  1.66e-01 & --     &  2.41e-01 & --     &  5.63e-02 & --   \\
            &   $40\times40$  &  9.94e-02 & 0.739  &  5.96e-02 & 2.017  &  8.39e-03 & 2.745  \\
            &   $80\times80$  &  5.64e-02 & 0.817  &  1.70e-02 & 1.810  &  9.15e-04 & 3.197  \\
            & $160\times160$  &  3.05e-02 & 0.886  &  4.74e-03 & 1.841  &  6.26e-05 & 3.871  \\
            & $320\times320$  &  1.63e-02 & 0.906  &  1.28e-03 & 1.886  &  4.76e-06 & 3.717  \\\hline 
		\end{tabular}
	\end{small}
	\caption{\label{tab2d:2}
		$L_{\infty}$-errors and orders of accuracy for  Example \ref{ex2d:2} at  $T=0.5/\pi^2$. }
\end{table}

\begin{figure}[b!]
	\centering
	\vspace{2mm}
	{\includegraphics[width=0.75\textwidth, height=4.3cm]{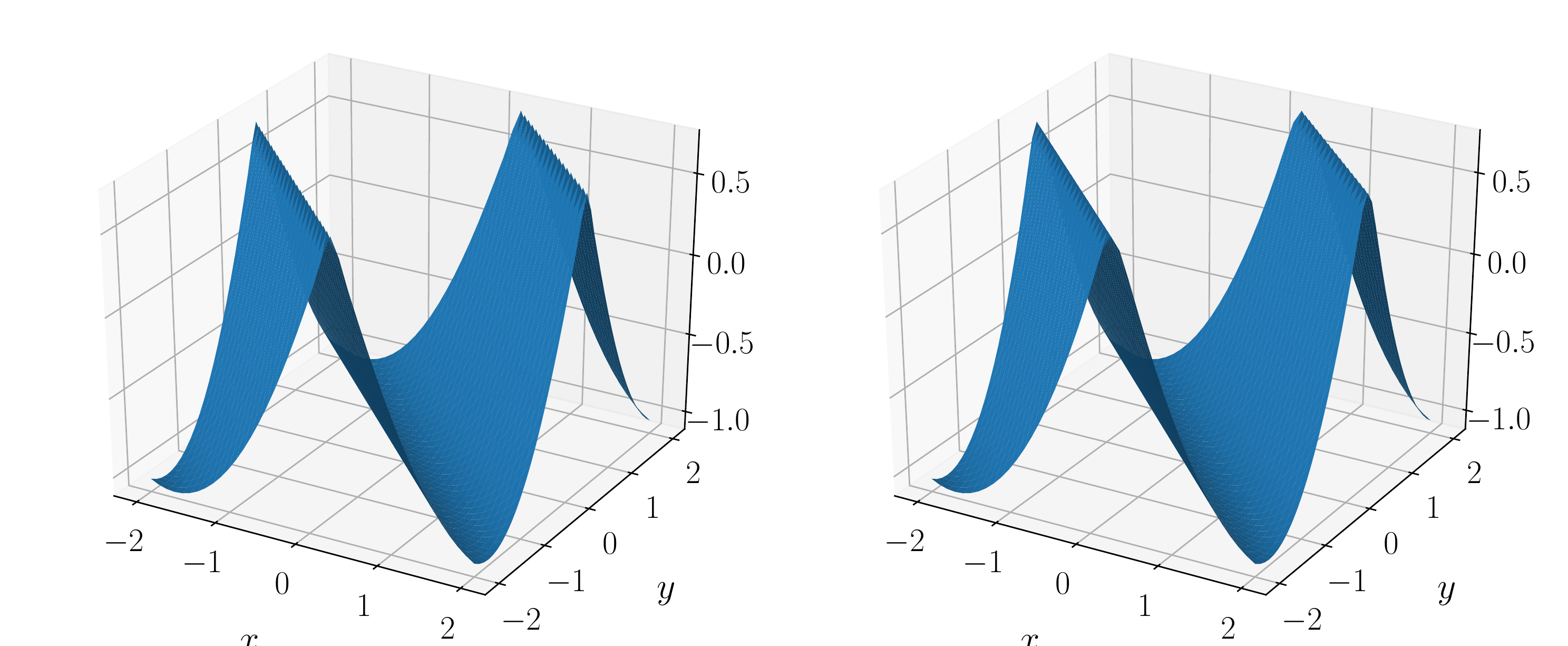}}
	\caption{\em  Numerical solutions for Example \ref{ex2d:2} with $\text{CFL}= 0.5$ (left) and $\text{CFL}= 2$ (right) at $T=1.5/\pi^2$.}
\label{fig2d:2}
\end{figure}

In this example, we solve the Burgers' equation
\begin{align}
    \phi_{t}+\frac{1}{2}(\phi_{x}+\phi_{y}+1)^2=0 
\end{align}
with a smooth initial condition
\begin{equation}\label{eq:2dinit}
    \phi(x,y,0)=-\cos(\pi(x+y)/2),
\end{equation}
on the periodic spatial domain $[-2,2]\times[-2,2]$. We compute the $L_\infty$ errors at time $T=0.5/\pi^2$, when the solution is smooth.
Table \ref{tab2d:2} shows the associated orders of accuracy measured for each of the schemes, which match the expected orders. Also in Figure \ref{fig2d:2}, is a plot of the numerical solutions using $40\times40$ grid points at time $T=1.5/\pi^2$, when the solution has developed a discontinuous derivative. 

\end{exa}

\begin{exa}\label{ex2d:3}
We now consider a Hamiltonian which is neither convex nor concave:
\begin{align}
    \phi_{t}-\cos(\phi_{x}+\phi_{y}+1)=0. 
\end{align}
Here the spatial domain  $[-2,2]\times[-2,2]$ is defined with periodic boundary conditions in both $x$ and $y$ directions. In Figure \ref{fig2d:3}, we plot the numerical solutions at time $T=1.5/\pi^2$ using a $40\times40$ grid with the smooth initial data \eqref{eq:2dinit}. We observe that the proposed schemes maintain good resolution when a larger CFL number is used, i.e., $\text{CFL}=2$, relative to $\text{CFL}=0.5$.

\begin{figure}[b!]
	\centering
	{\includegraphics[width=0.75\textwidth, height=4.3cm]{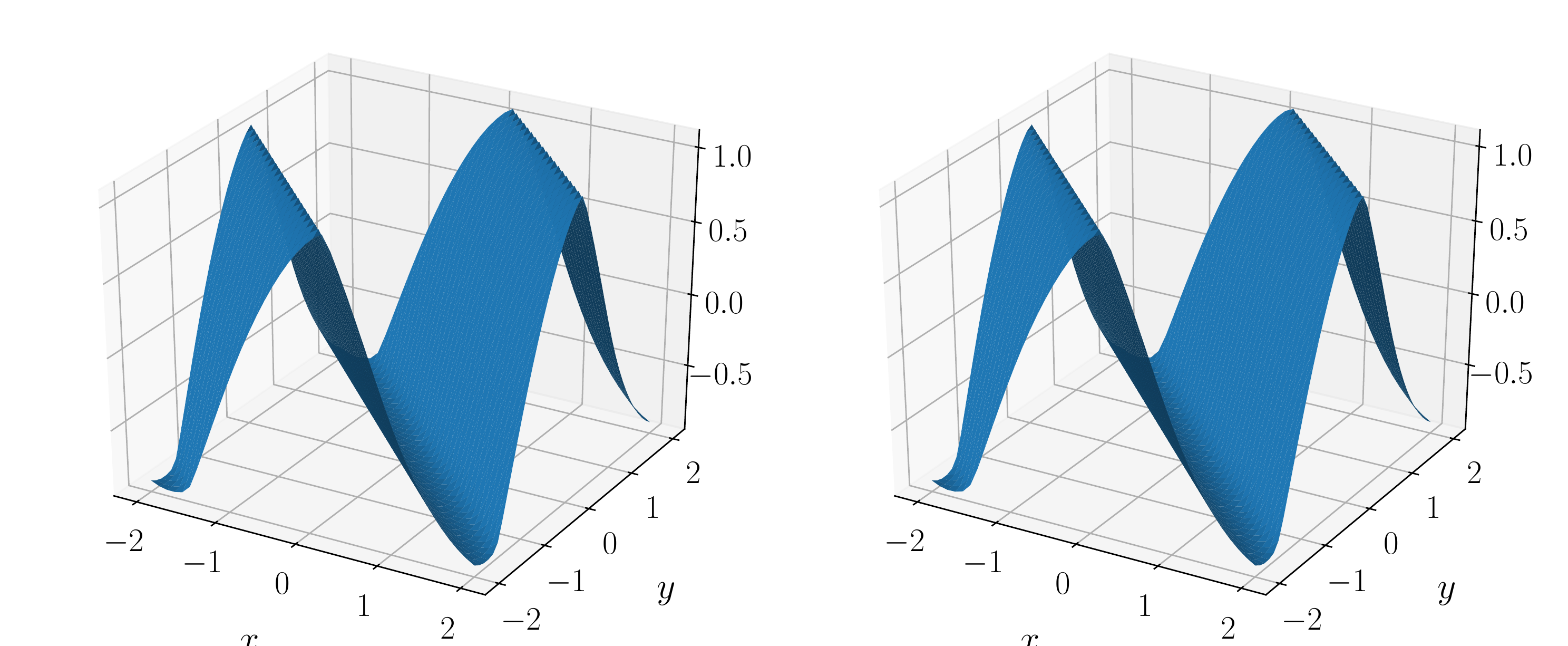}}
	\caption{\em  Numerical solutions for Example \ref{ex2d:3} with CFL=$0.5$ (left) and CFL=$2$ (right) at $T=1.5/\pi^2$.}
	\label{fig2d:3}
\end{figure}
\end{exa}

\begin{exa}\label{ex2d:5}
\begin{figure}[b!]
	\centering
\subfigure[CFL = $0.5$]{\includegraphics[width=0.78\textwidth, height=4.2cm]{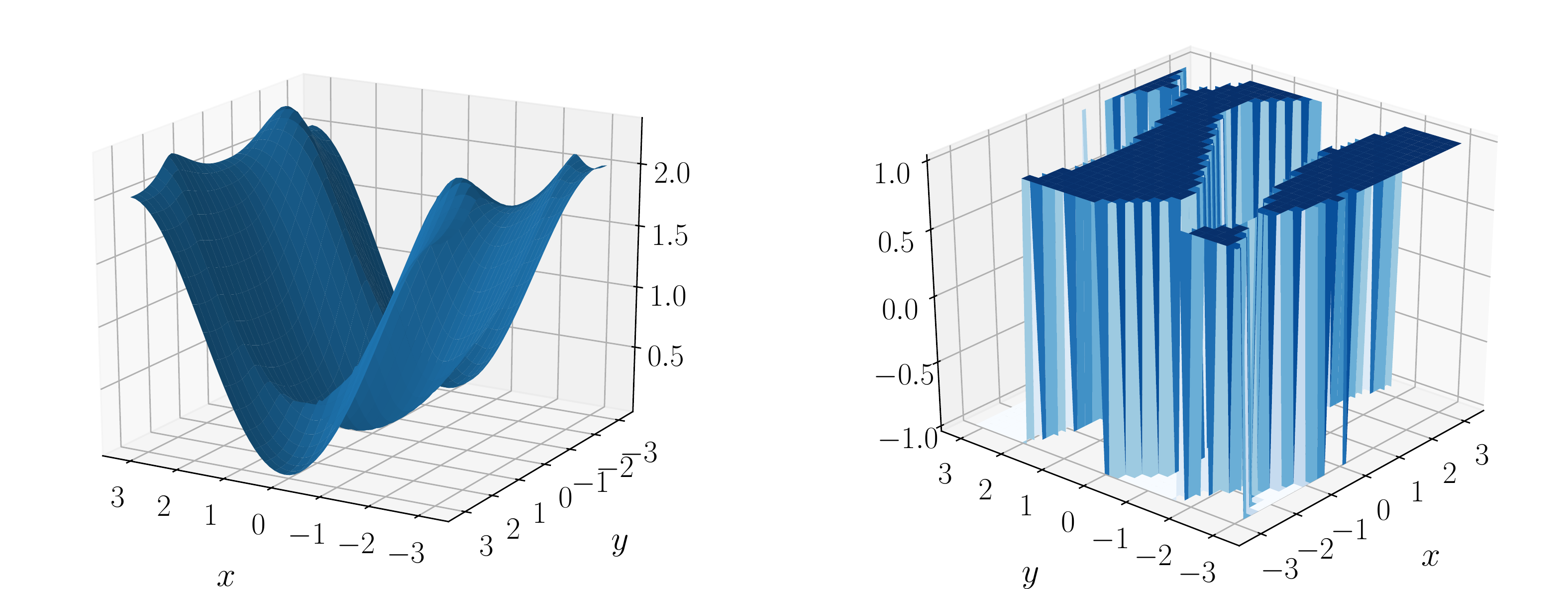}}
\subfigure[CFL = $2$]{\includegraphics[width=0.78\textwidth, height=4.2cm]{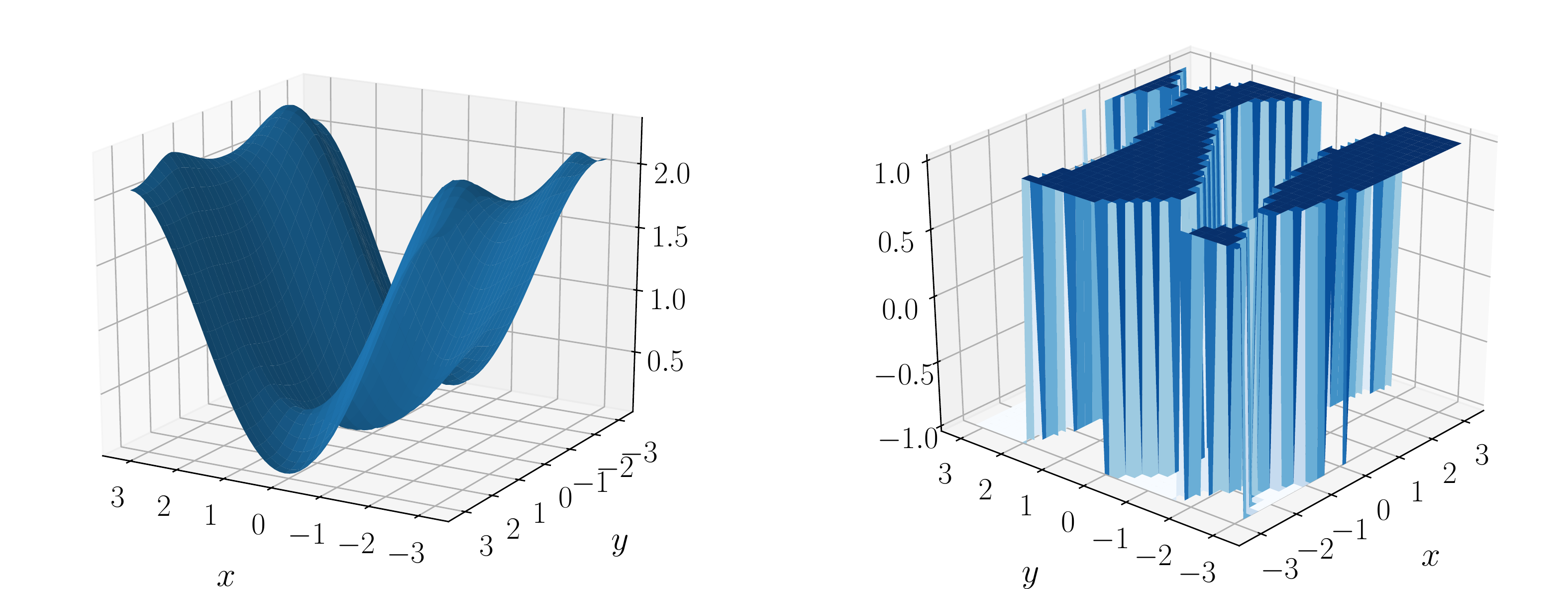}}
	\caption{\em Numerical solutions' surfaces (left) and optimal controls (right) for Example \ref{ex2d:5} at $T=1$.}
	\label{fig2d:5} 
\end{figure}

We solve the following optimal control problem related to cost determination:
\begin{align}
\begin{split}
&\phi_{t} + \sin(y) \phi_{x}+(\sin(x)+sign(\phi_{y}))\phi_{y}-\frac{1}{2}\sin^2(y)-1+\cos(x) =0,\\
&\phi(x,y,0)=0,\\
\end{split}
\end{align}
on the periodic spatial domain $[-\pi,\pi]\times[-\pi,\pi]$. We compute the numerical solutions up to $T=1$ using grids of size $60\times60$ and provide plots of the numerical solution and the optimal control $sign(\phi_y)$ in Figure \ref{fig2d:5}. Proposed schemes capture the non-smooth structures of the solutions with both CFL $ = 0.5$ and $2$.
\end{exa}

\begin{exa}\label{ex2d:4}
\begin{figure}[b!]
	\centering
	\subfigure[Proposed (CFL=$0.5$)]{\includegraphics[width=0.7\textwidth]{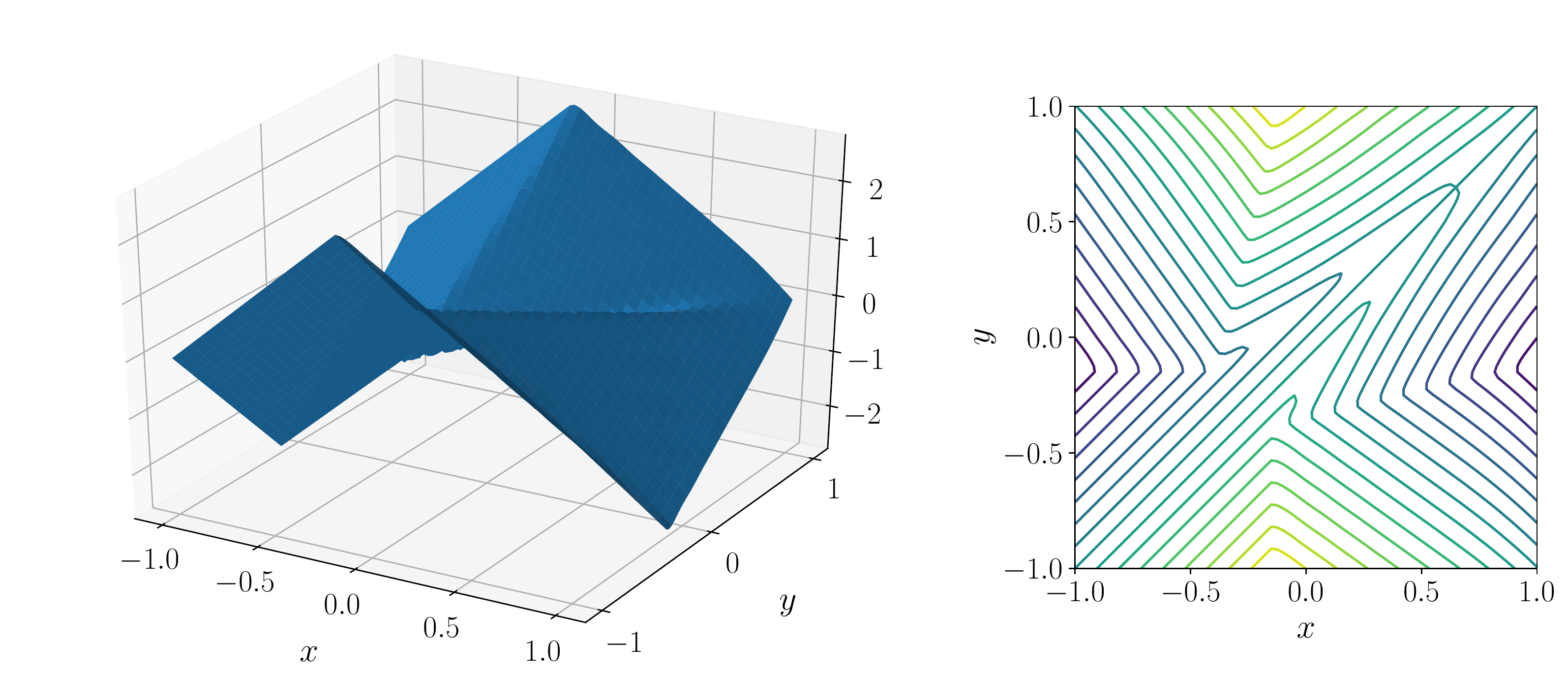}}
	\subfigure[Proposed (CFL=$2$)]{\includegraphics[width=0.7\textwidth]{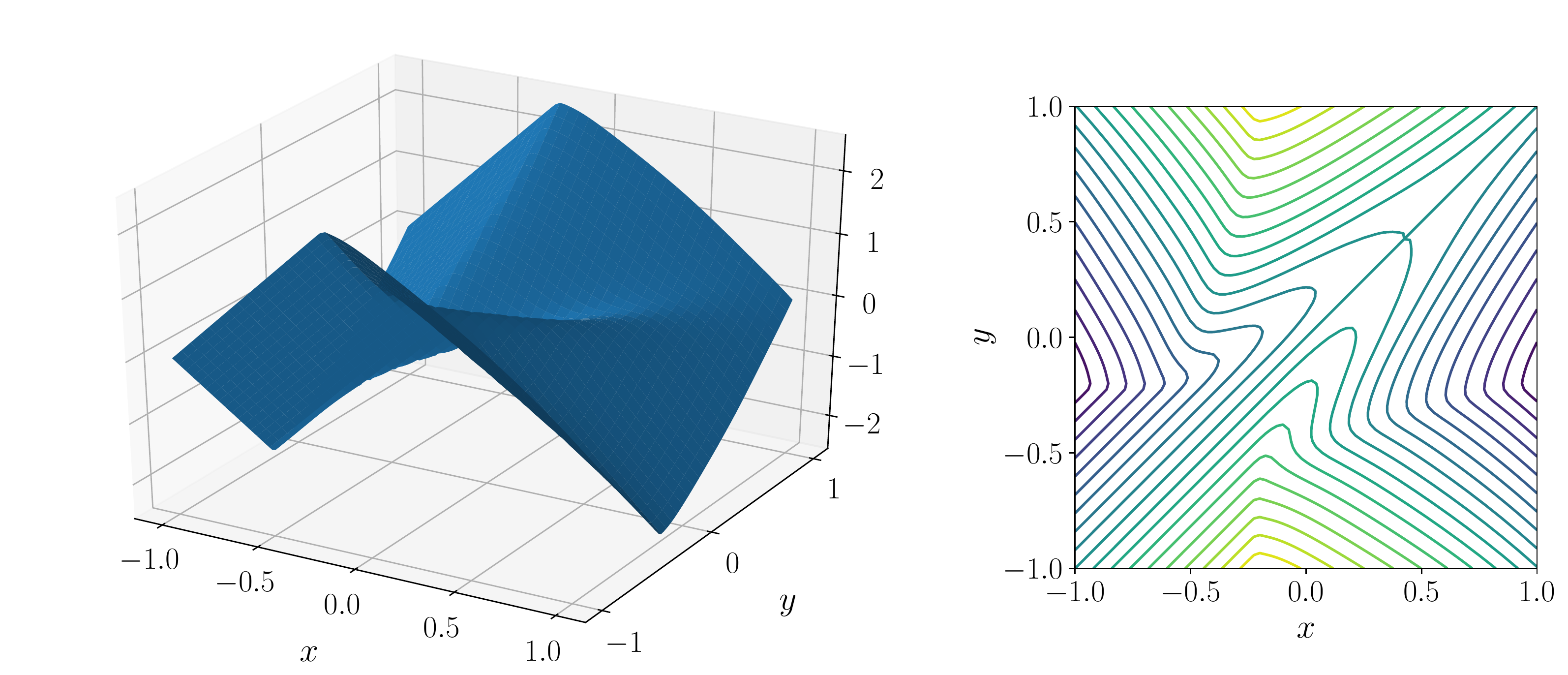}}
	\subfigure[Previous (CFL=$2$)]{\includegraphics[width=0.7\textwidth]{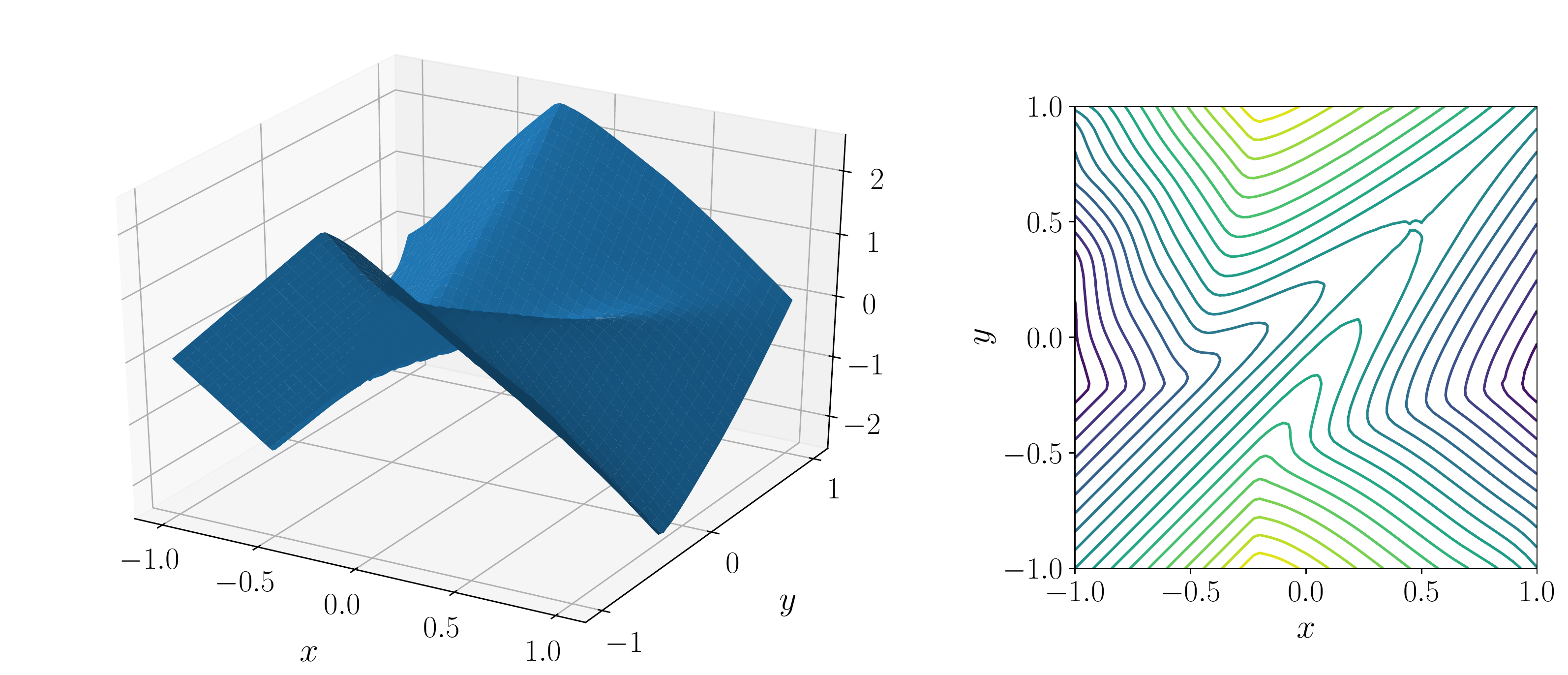}}
	\caption{\em Numerical solutions' surfaces (left) and contour (right) for Example \ref{ex2d:4} at $T=1$.}
	\label{fig2d:4}	
\end{figure}

In this example, we consider the two-dimensional Riemann problem with a non-convex Hamiltonian:
\begin{align}
\phi_{t} + \sin( \phi_{x}+\phi_{y} ) = 0,
\end{align}
on the spatial domain $[-1,1]\times[-1,1]$, where outflow boundary conditions are imposed. Our initial data is represented by the non-smooth function
\begin{align*}
    \phi(x,y,0)=\pi (|y|-|x|).
\end{align*}
We run the simulation using $80\times 80$ grid points and track the solution up to time $T=1$. Plots of the numerical solutions are provided in Figure \ref{fig2d:4}. Note that when a CFL number of $2$ is used, the proposed scheme reduces dissipation encountered by the previous scheme near the boundaries.
\end{exa}

\begin{exa}\label{ex2d:6}
	The next problem is a prototypical model in geometric optics, which is a Cauchy problem for H-J equation with a non-convex Hamiltonian:
		\begin{align}
		\begin{split}
		&  \phi_t + \sqrt{\phi_x^2 + \phi_y^2 + 1} = 0, \\
		&  \phi(x,y,0) = 0.25(\cos(2\pi x) - 1)(\cos(2\pi y) - 1)- 1,
		\end{split}
		\end{align}
		on a periodic domain $[0,1]\times[0,1]$. We approximate the solution using $60 \times 60$ grids up to a final time $T= 0.6$, during which the characteristics intersect. The surfaces and contour lines of the numerical solution are shown in Figure \ref{fig2d:6}. We note that sharp and symmetric regions are well-maintained in the solutions. 
\end{exa}

\begin{figure}[b!]
	\centering
	\subfigure[CFL = $0.5$]{\includegraphics[width=0.75\textwidth]{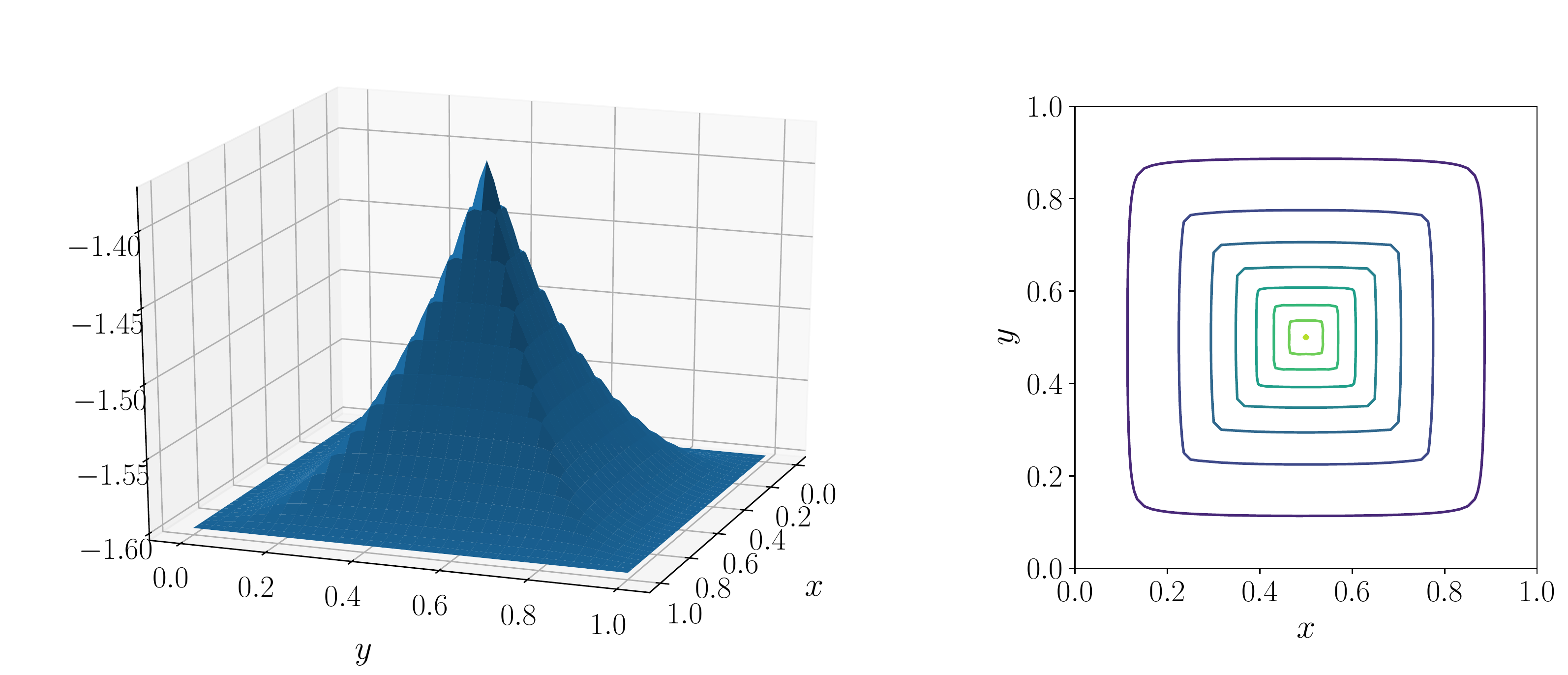}}
	\subfigure[CFL = $2$]{\includegraphics[width=0.75\textwidth]{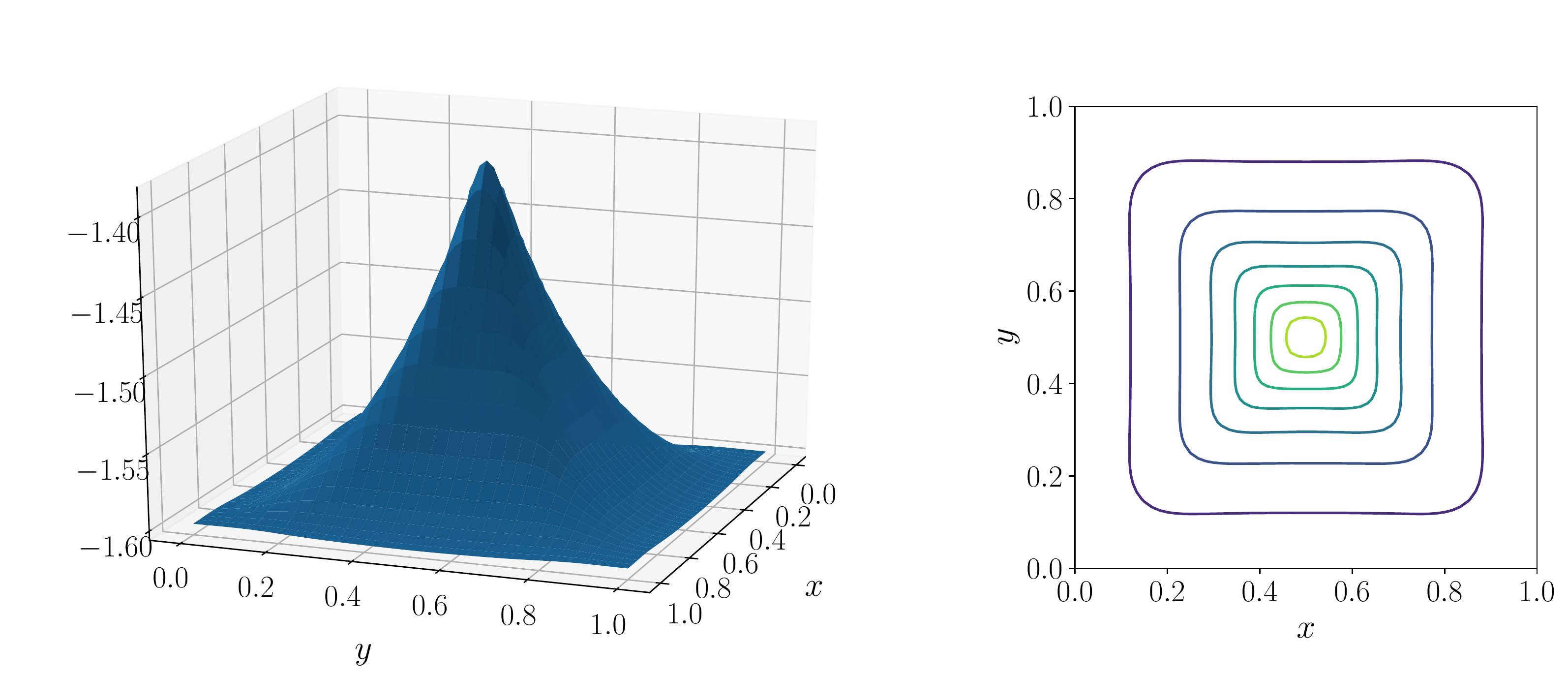}}
	\caption{\em Numerical solutions' surfaces (left) and contours (right) for Example \ref{ex2d:6} at $T=0.6$.}
	\label{fig2d:6}
\end{figure}


\begin{exa}\label{ex2d:7}
In this test problem, we change the sign of Hamiltonian and use the same initial function as the previous example to simulate a propagating surface: 
\begin{align}\label{eq:propagation}
\begin{split}
&  \phi_t - \sqrt{\phi_x^2 + \phi_y^2 + 1} = 0, \\
&  \phi(x,y,0)  = 1- 0.25(\cos(2\pi x) - 1)(\cos(2\pi y) - 1),
\end{split}
\end{align}
on the periodic domain $[0,1]\times[0,1]$.
As above, $60 \times 60$ grid points are used and 
the snapshots of numerical solutions at $t=0,\,0.3,\,0.6$ and $0.9$ are given in Figures \ref{Fig8}. We have also included plots of the solutions, which do not use the nonlinear filters (see Figures 4.\ref{Fig8.3}) to demonstrate their effect.

\begin{figure}[htb]
	\centering
    \vspace{2mm}
	\subfigure[CFL=$0.5$ ]{
		\includegraphics[width=0.31\textwidth]{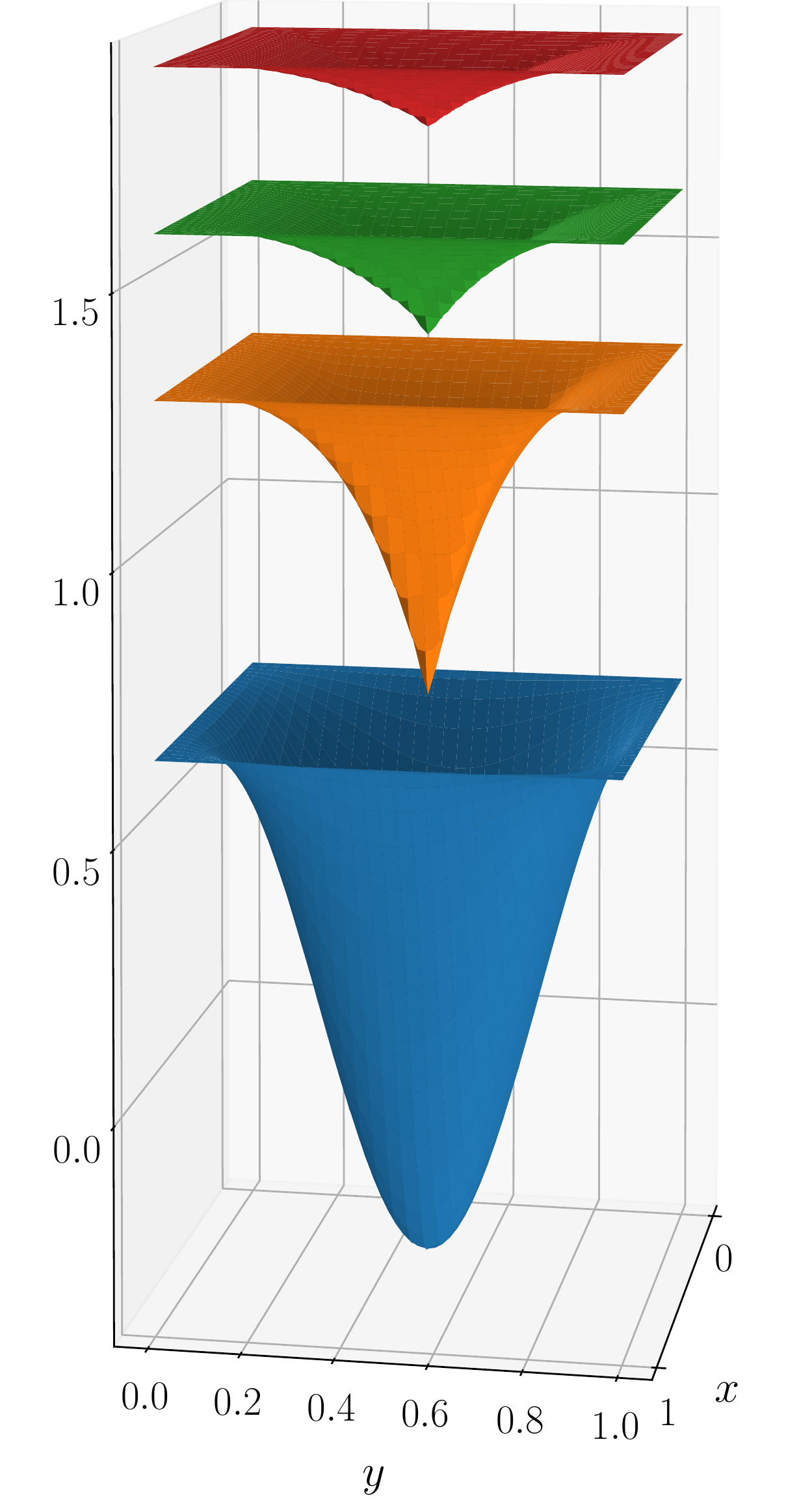}}
	\subfigure[CFL=$2$ with filters]{
		\includegraphics[width=0.31\textwidth]{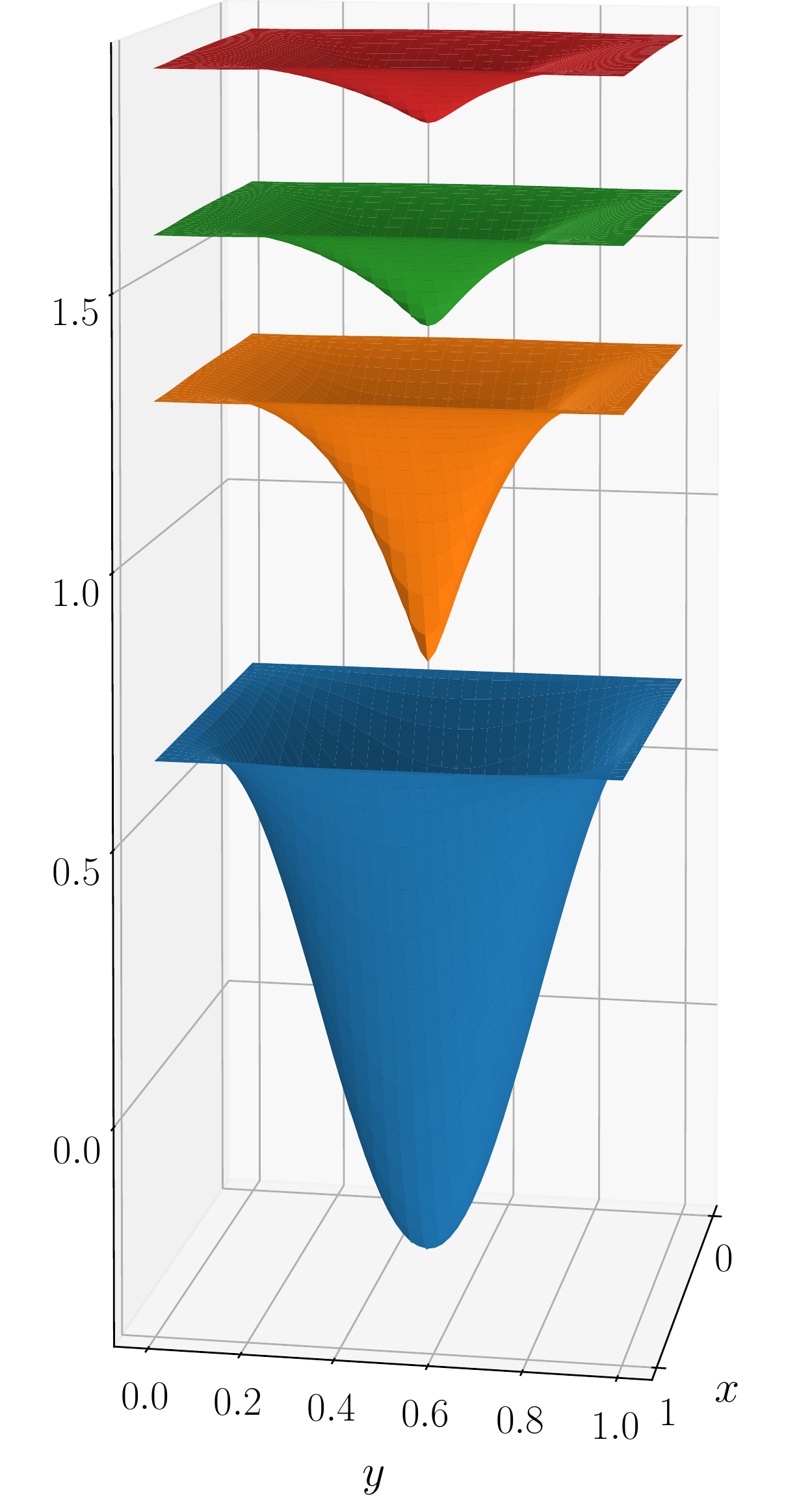}}
	\subfigure[CFL=$2$ without filters]{
	\includegraphics[width=0.31\textwidth]{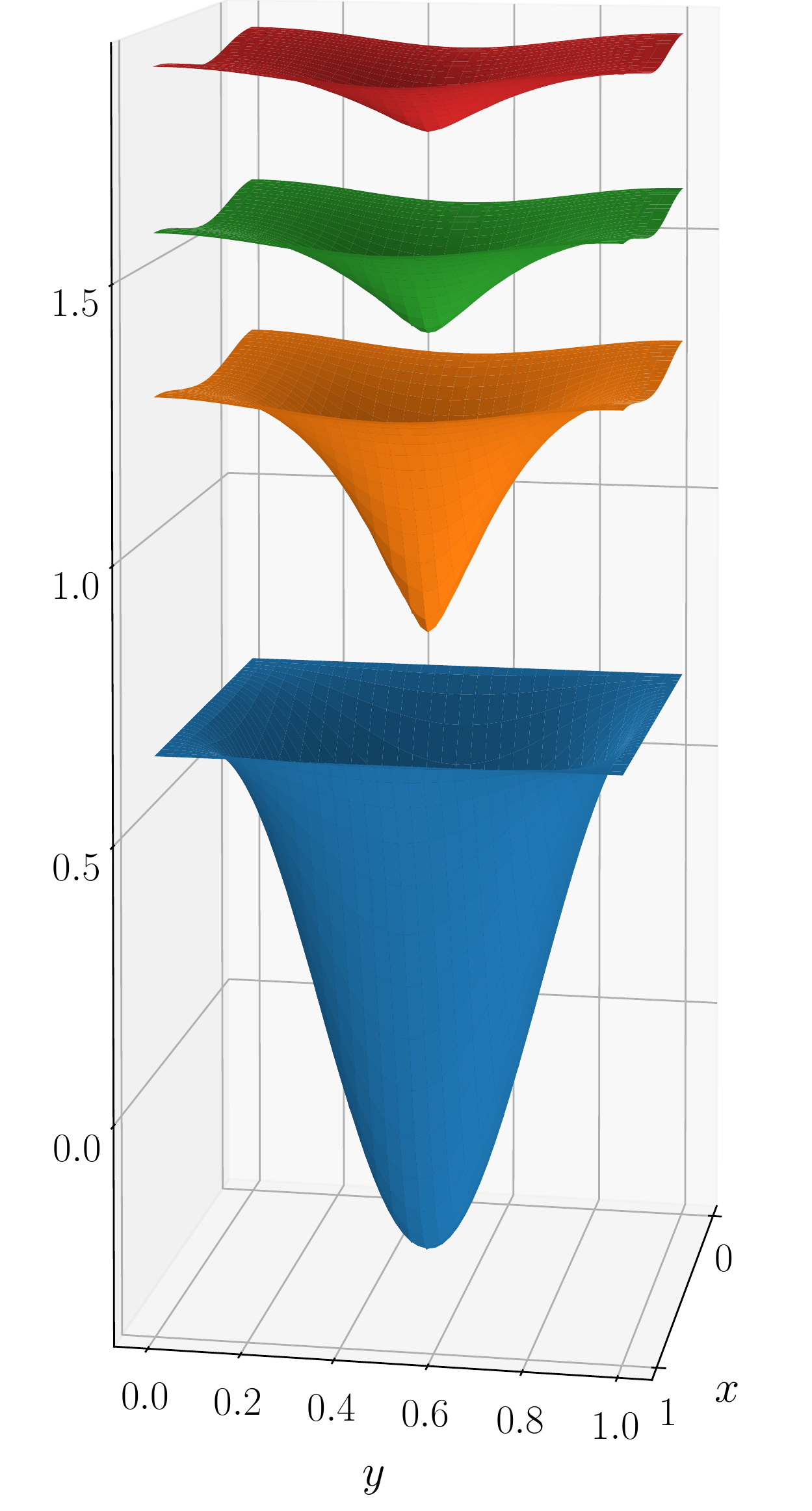}\label{Fig8.3}}
	\caption{\em Propagating numerical solutions' surfaces for Example \ref{ex2d:7} at T=0, 0.3, 0.6 and 0.9}
	\label{Fig8}
    \vspace{2mm}
\end{figure}

\end{exa}

\subsection{Examples with irregular grids}\label{subsec:map}

In this section, we apply the proposed scheme to several examples defined on irregular grids $x$ and $y$ in order to demonstrate the capabilities of the coordinate transformation $x=x(\xi), y=y(\eta)$ using uniform computational grids $\xi$ and $\eta$. Here, we present the results with the time step
\begin{align}
\Delta t=\frac{\text{CFL}}{\max(\alpha_{x}/\Delta \xi,\alpha_{y}/\Delta \eta)},
\end{align}
where $\alpha_x$ and $\alpha_y$ are the maximum wave propagation speeds in the $x$ and $y$ directions, respectively.

\begin{exa}\label{ex_map}
\begin{table}[t!]
	\centering
	\begin{small}
		\begin{tabular}{|c|c|cc|cc|cc|}
			\hline
			\multirow{2}{*}{CFL} &  \multirow{2}{*}{$N_x\times N_y$} & \multicolumn{2}{c|}{$k=1$. $\beta=1$.} & \multicolumn{2}{c|}{$k=2$. $\beta=0.5$.} & \multicolumn{2}{c|}{$k=3$. $\beta=0.6$.}\\
			\cline{3-8}
			& &  error &   order  &  error &  order  &  error  & order  \\\hline	
			\multirow{5}{*}{0.5}  
			&   $40\times40$  &  1.09e-01 & --    &  6.43e-02 & --     &  7.76e-03 & --    \\
			&   $80\times80$  &  6.14e-02 & 0.835 &  1.66e-02 & 1.951  &  1.00e-03 & 2.952 \\
			& $160\times160$  &  3.30e-02 & 0.895 &  4.47e-03 & 1.894  &  7.88e-05 & 3.669 \\
			& $320\times320$  &  1.71e-02 & 0.945 &  1.19e-03 & 1.916  &  4.69e-06 & 4.069 \\
			& $640\times640$  &  8.72e-03 & 0.975 &  3.05e-04 & 1.959  &  3.36e-07 & 3.802 \\\hline                   	
			\multirow{5}{*}{1} 
			&   $40\times40$  &  1.73e-01 & --     &  1.98e-01 & --    &  4.67e-02 & -- \\
			&   $80\times80$  &  1.09e-01 & 0.657  &  6.45e-02 & 1.621 &  7.72e-03 & 2.598  \\
			& $160\times160$  &  6.15e-02 & 0.831  &  1.66e-02 & 1.961 &  9.92e-04 & 2.960  \\
			& $320\times320$  &  3.30e-02 & 0.898  &  4.46e-03 & 1.893 &  7.13e-05 & 3.799  \\
			& $640\times640$  &  1.71e-02 & 0.945  &  1.18e-03 & 1.915 &  6.97e-06 & 3.354  \\\hline 
			\multirow{5}{*}{2}  
			&   $40\times40$  &  2.16e-01 & --     &  3.21e-01 & --     &  1.73e-01 & -- \\
			&   $80\times80$  &  1.72e-01 & 0.327  &  1.98e-01 & 0.696  &  4.68e-02 & 1.888  \\
			& $160\times160$  &  1.10e-01 & 0.653  &  6.45e-02 & 1.621  &  7.72e-03 & 2.600  \\
			& $320\times320$  &  6.15e-02 & 0.835  &  1.66e-02 & 1.962  &  1.00e-03 & 2.948  \\
			& $640\times640$  &  3.30e-02 & 0.897  &  4.46e-03 & 1.894  &  7.62e-05 & 3.716  \\\hline 
		\end{tabular}
	\end{small}
	\caption{\label{tab_map}\em $L_{\infty}$-errors and orders of accuracy for Example \ref{ex_map}  at $T=0.5/\pi^2$. }
\end{table}

	We first consider the two-dimensional Burgers' equation
	\begin{align}
	\begin{split}
	&\phi_{t}+\frac{1}{2}(\phi_{x}+\phi_{y}+1)^2=0, \\
	&\phi(x,y,0)=-\cos(\pi(x+y)/2),\\
	\end{split}
	\end{align}
	on the periodic domain $[-2,2]\times[-2,2]$. We generate a nonuniform mesh using random perturbations and compute the $L_\infty$ errors and orders of accuracy at $T=0.5/\pi^2$ while the solution is still smooth. In Table \ref{tab_map}, we confirm the convergence rates of the mapped scheme on nonuniform meshes.
\end{exa}


\begin{exa}\label{ex6-map}
	 We solve the two-dimensional Riemann problem with a non-convex Hamiltonian
	\begin{align}
	\begin{split}
	& \phi_{t} + \sin( \phi_{x}+\phi_{y} )=0, \quad -1\leq x,y\leq 1\\
	& 	\phi(x,y,0)=\pi (|y|-|x|)\\
	\end{split}
	\end{align}
	which we considered in Example \ref{ex2d:4}. To see the efficiency of the scheme, we construct a nonuniform mesh consisting of $60\times60$ grid points using a geometric series, selecting the ratio between the smallest cell size and the biggest cell size to be $1:7$. The resulting mesh is displayed in Figure 4.\ref{fig-mesh1}. On this nontrivial grid, we plot the numerical solutions' surfaces and contour lines at time $T=1$ in Figure \ref{Fig6map}. 
    \begin{figure}[h!]
    	\centering
    	\subfigure[Nonuniform mesh]{
    	\includegraphics[width=0.3\textwidth]{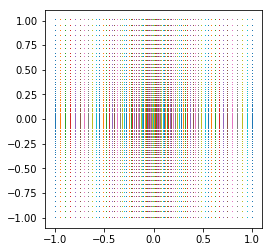}    	\label{fig-mesh1} }
    	\subfigure[Numerical solutions' surfaces (left) and contour (right)]
    	{\includegraphics[width=0.73\textwidth]{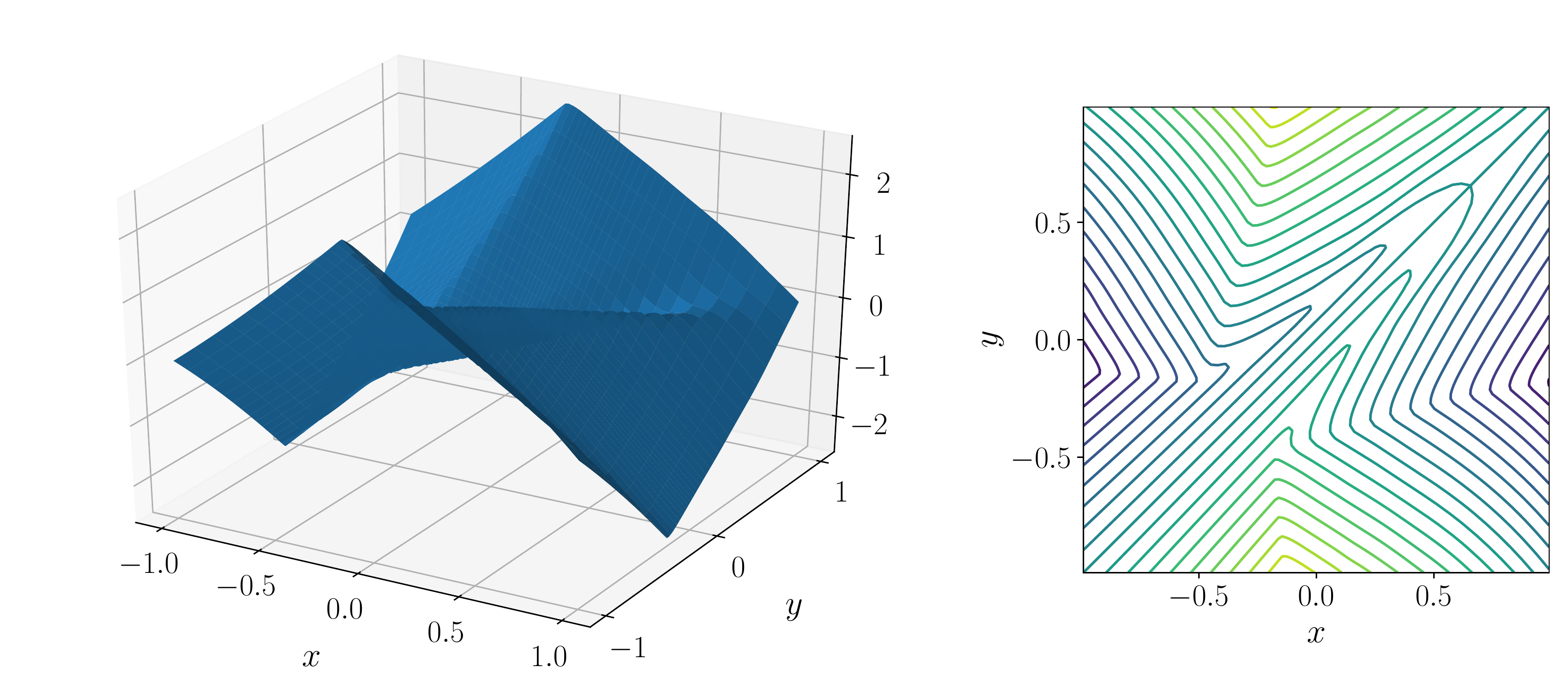}    	}
    	\caption{\em Numerical solutions on nonuniform meshes for Example \ref{ex6-map} at $T=1$.}
    		\label{Fig6map}
    \end{figure}

\end{exa}


\begin{exa}\label{ex9}

    \begin{figure}[b!]
        \centering
        \subfigure[The discretization of the domain]
                {\includegraphics[width=0.3\textwidth,clip=false]{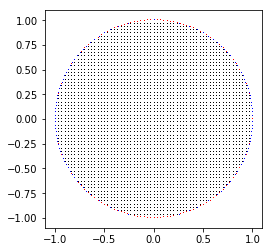}    	\label{Fig9a}}
                
    	\vspace{2mm}
        \subfigure[Propagating numerical solutions with CFL=$0.5$ (Left) and CFL=$2$ (Right)]{
    	{\includegraphics[width=0.31\textwidth, height=10.3cm]{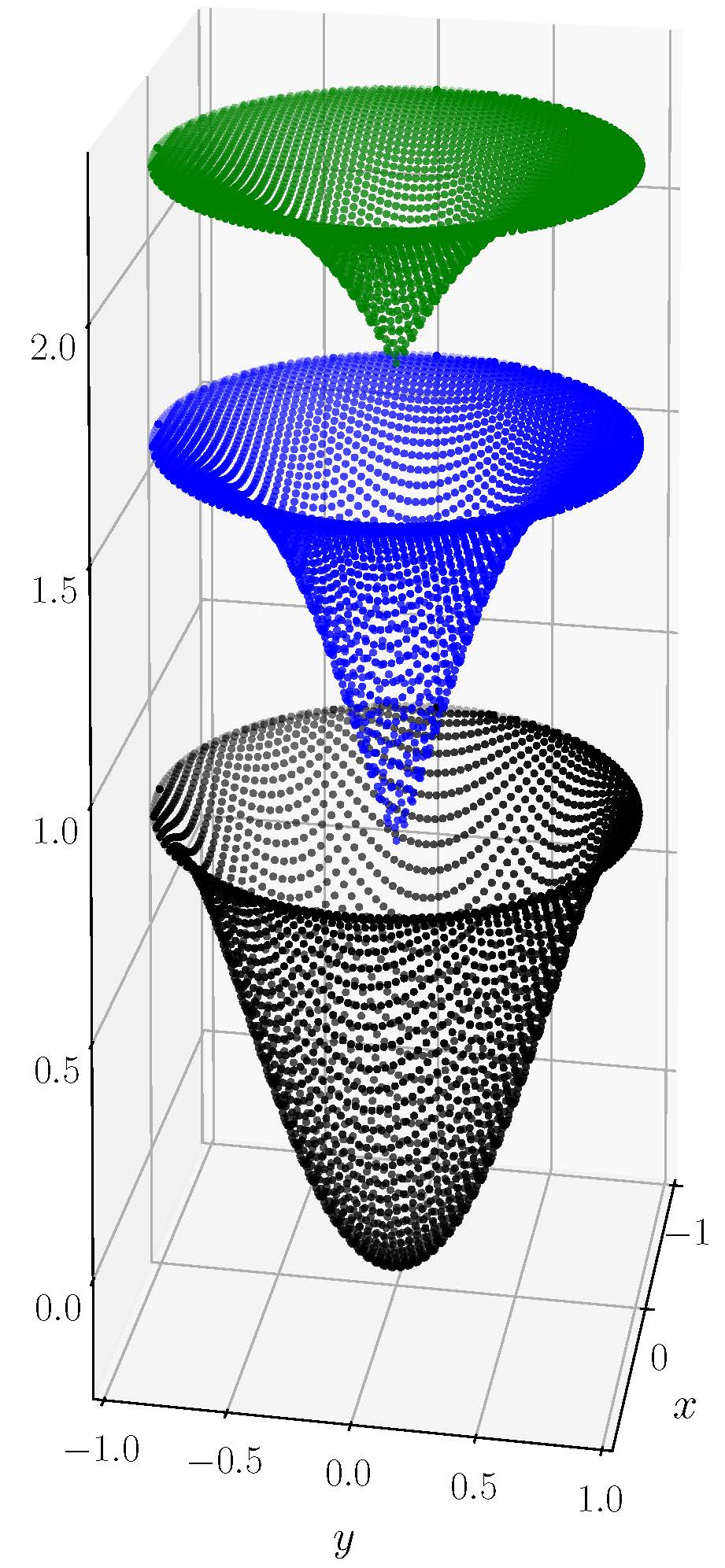}}
    	\hspace{7mm}
    	{\includegraphics[width=0.31\textwidth, height=10.3cm]{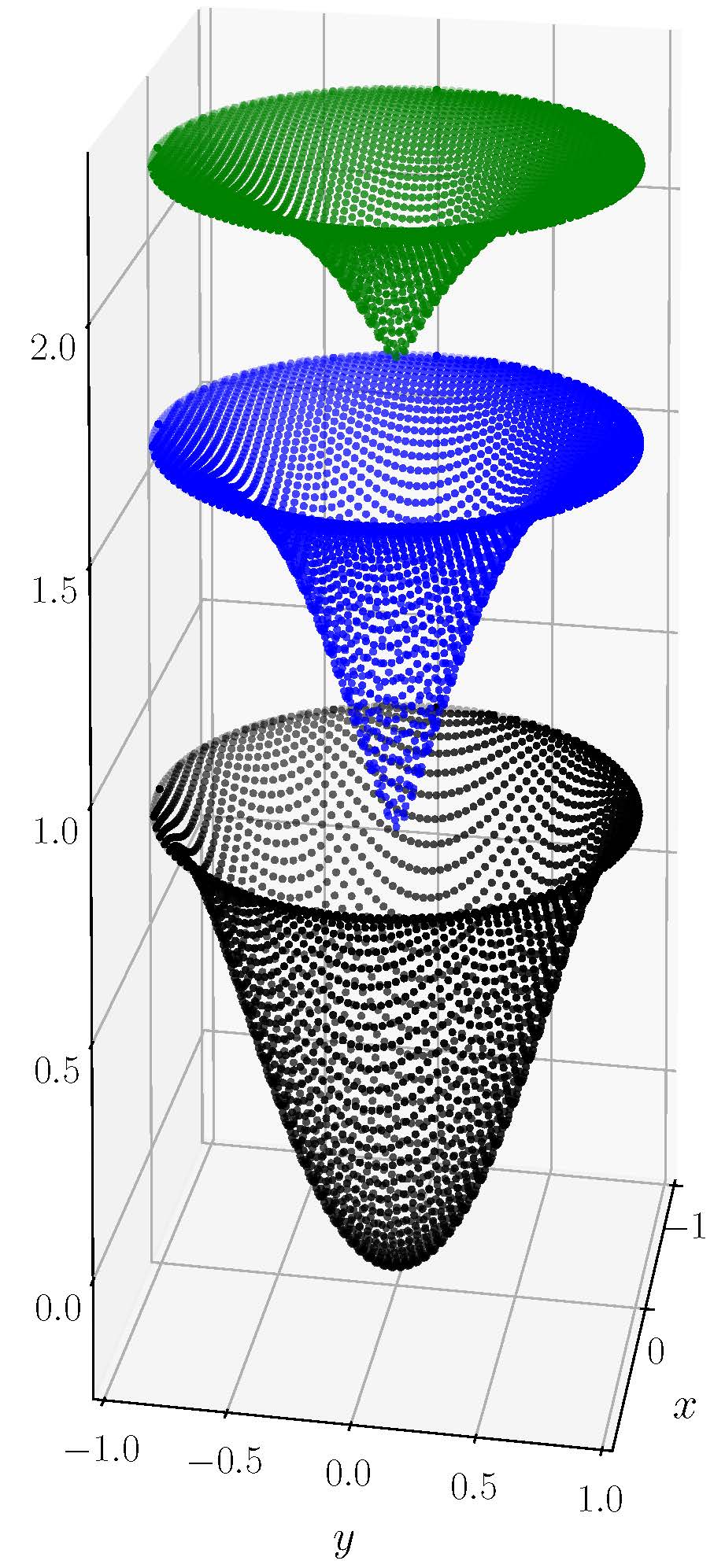}}}
    	\caption{\em Propagating numerical solutions at T=0, 0.6 and 1.2 on nontrivial meshes for Example \ref{ex9}}
    	\label{Fig9b}
    \end{figure}
    
	We consider the same problem of a propagating surface \eqref{eq:propagation} in Example \ref{ex2d:7} 
	with the initial condition
	$$\phi(x, y, 0) = \sin \left( \frac{\pi}{2}(x^2+y^2) \right) $$ defined on the unit disk $x^2 + y^2 \leq 1$ where the Dirichlet boundary condition $$\phi(x,y,t)=1+t \quad \text{for all}\quad x^2+y^2=1$$
	is imposed. The domain is discretized by embedding the boundary in a regular $60\times 60$ Cartesian mesh so that the irregular spacing occurs only near the boundary. The discretization of the domain is illustrated in Figure 4.\ref{Fig9a}. Blue and Red dots in the Figure indicate $x$ and $y$ directional boundary points, respectively. Snapshots of the propagating surface taken at $T=0, 0.6$ and $1.2$ are given in Figure \ref{Fig9b}.
\end{exa}


\begin{exa}\label{ex10}
	As our last example on nonuniform meshes, we consider the ``level set reinitialization'' equation \cite{sus}
	\begin{align}
    	\phi_t + {\rm sign}(\phi_0) (\sqrt{ \phi_x^2 + \phi_y^2} -1 ) = 0,
	\end{align}
	on the circular domain $\frac{1}{2} < \sqrt{x^2+y^2} < 1$. We choose an initial function with the signed distance function to the circle centered
	at the origin:
	\begin{equation}
	    \phi(x,y,0) = \phi_0 (x,y) = \sqrt{x^2+y^2} - 0.5.
	\end{equation}
	The hole in the circular domain is discretized by, again, embedding the boundary in a regular Cartesian mesh with $60\times60$ grid points. In Figure \ref{Fig10a}, we plot the resulting surface of the numerical solution at time $T=1$ on the mesh.

\begin{figure}[h!]
	\centering
	\subfigure[The domain]{	\includegraphics[width=0.3\textwidth]{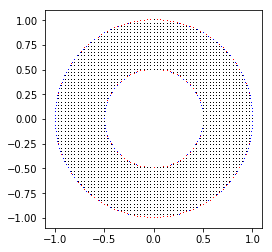} }
	\hspace{3mm}
	\subfigure[Solution]{\includegraphics[width=0.4\textwidth]{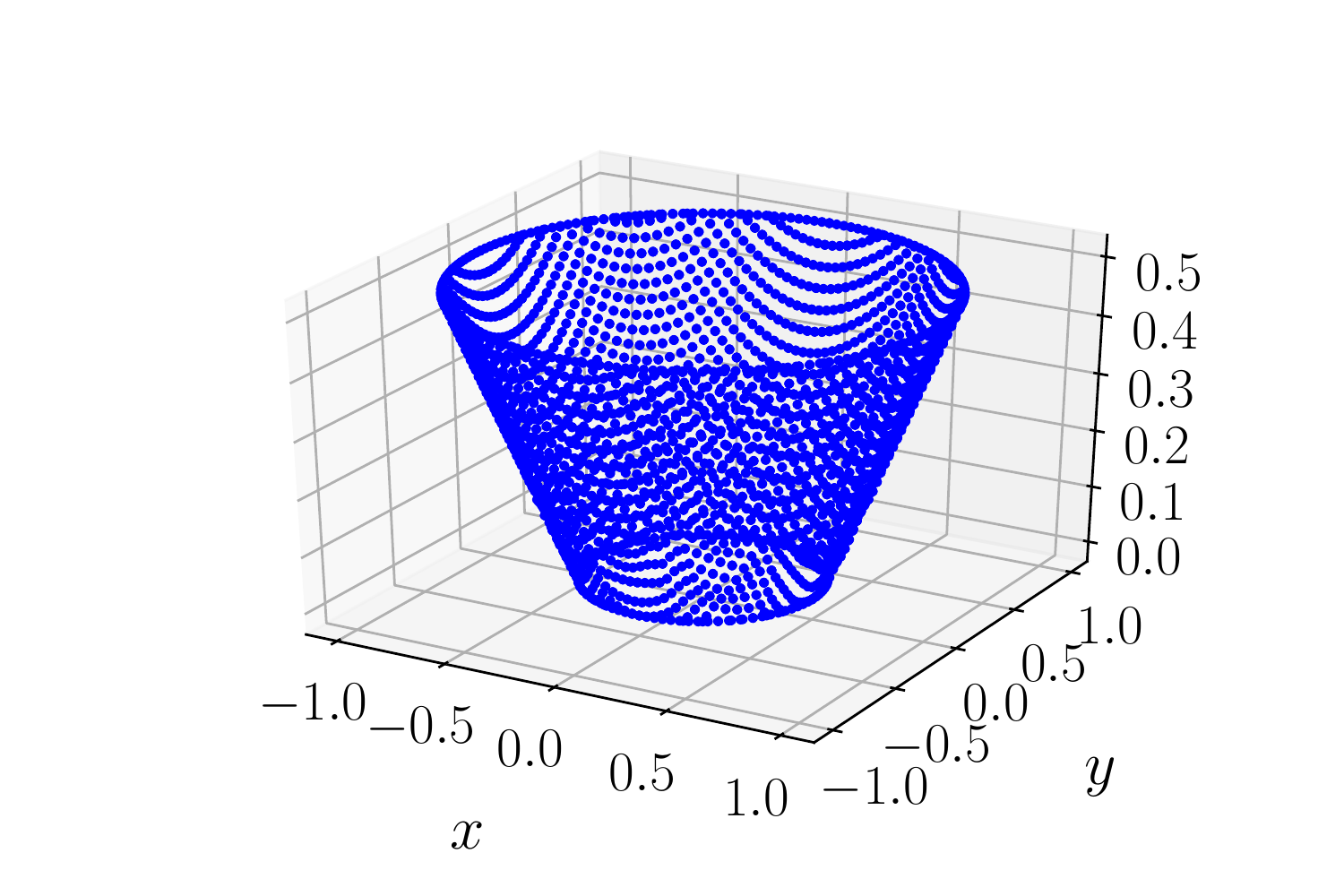}}
	\caption{\em (a) The discretization of the domain which is embedded in a Cartesian grid. Blue and Red dots indicate $x$ and $y$ directional boundary points, respectively. (b) Surface of the solution for Example \ref{ex10}.}
	\label{Fig10a}
\end{figure}

\end{exa}


\section{Conclusion}

In this paper, a class of high order unconditionally stable scheme is proposed to solve the Hamilton-Jacobi (H-J) equations.
Following the previous works \cite{christlieb2017kernel} and \cite{christlieb2019kernel}, the scheme is developed in the light of the kernel based approach for approximating the spatial derivatives in the H-J equation. 
The proposed scheme makes use of a  exponential basis to construct a novel WENO methodology for the space discretization.   The new WENO method is adept at  capturing sharp gradients without oscillations. By leveraging a coordinate transformation, we implement this scheme on high dimensional nonuniform meshes to enable the method to compute numerical solutions defined on domains containing more complex geometry.  The new method outperforms our previous unconditionally stable kernel based method on all examples tried in this work.

Although one of the advantages of using exponential polynomials (e.g., $e^{\lambda x}$) is that they can be tuned by choosing a tension parameter $\lambda$ that depend on the characteristics of the given data , the choice of the parameter is not a focus of this paper. For our  future study, we would like to develop a WENO scheme capable of choosing the optimal parameters in the local space.

\bibliographystyle{abbrv}
\bibliography{ref}

\end{document}